\documentclass{amsart}
\usepackage[margin=1in]{geometry}
\usepackage[numbers]{natbib}
\usepackage{amsmath,amssymb,amsaddr}
\usepackage[shortlabels]{enumitem}
\usepackage{mathtools}
\usepackage{hyperref}
\usepackage{algorithm}
\usepackage{algpseudocode}
\usepackage{bm}
\usepackage{mdframed}
\usepackage{booktabs}
\usepackage{graphicx}
\usepackage{stmaryrd}
\usepackage{xcolor}

\graphicspath{{./Figures/}}

\def\R{\mathbb{R}}
\def\T{\mathcal{T}}
\def\O{\mathcal{O}}

\def\Q{\mathcal{Q}}
\def\oned{\mathrm{1D}}
\def\ndof{n_{\textit{dof}}}
\def\nel{n_{\textit{el}}}
\def\ncoarse{n_{\textit{coarse}}}
\def\npatch{n_{\textit{loc}}}
\def\npatches{J}

\def\DG{\textit{DG}}
\def\IP{\textit{IP}}
\def\BR{\textit{BR2}}

\def\llb{\llbracket}
\def\rrb{\rrbracket}

\newtheorem{lem}{Lemma}
\newtheorem{thm}{Theorem}
\newtheorem{cor}{Corollary}
\newtheorem{prop}{Proposition}
\newtheorem{rem}{Remark}

\title
  [Low-order preconditioners for high-order CG and DG]
  {Efficient low-order refined preconditioners for high-order matrix-free
  continuous and discontinuous Galerkin methods}
\author{Will Pazner}
\address{Center for Applied Scientific Computing, Lawrence Livermore National
       Laboratory}

\begin{document}

\begin{abstract}
   In this paper, we design preconditioners for the matrix-free solution of high-order continuous and discontinuous Galerkin discretizations of elliptic problems based on FEM-SEM equivalence and additive Schwarz methods.
   The high-order operators are applied without forming the system matrix, making use of sum factorization for efficient evaluation.
   The system is preconditioned using a spectrally equivalent low-order ($p=1$) finite element operator discretization on a refined mesh.
   The low-order refined mesh is anisotropic and not shape regular in the polynomial degree of the high-order operator, requiring specialized solvers to treat the anisotropy.
   We make use of an element-structured, geometric multigrid V-cycle with ordered ILU(0) smoothing.
   The preconditioner is parallelized through an overlapping additive Schwarz method that is robust in $h$ and $p$.
   The method is extended to interior penalty and BR2 discontinuous Galerkin discretizations, for which it is also robust in the size of the penalty parameter.
   Numerical results are presented on a variety of examples, verifying the uniformity of the preconditioner.
\end{abstract}

\maketitle

\section{Introduction}

High-order numerical methods are playing an increasingly significant role in many areas of scientific computation in recent years \cite{Wang2013,Slotnick2014,Brown2018}.
These methods promise higher accuracy with fewer degrees of freedom, at the cost of more arithmetic operations performed per degree of freedom.
Because of their high arithmetic intensity, high-order methods have been seen as promising candidates to run on GPUs and accelerator-based architectures \cite{Klockner2009,Vermeire2017}.
Due to the restrictive memory limitations on these architectures, much of the past research in this area has focused on problems with explicit time integration, thus avoiding the solution of large linear systems of equations.
In this work, we study the iterative solution to the linear systems arising from high-order finite element discretizations in the matrix-free context, avoiding the restrictive memory cost of assembling the system matrix.

A standard finite element or discontinuous Galerkin method with polynomial degree $p$ will result in coupling between all the degrees of freedom within a single element (and with a subset of the degrees of freedom of neighboring elements).
The number of degrees of freedom per element scales like $\mathcal{O}(p^d)$ in $d$ spatial dimensions, and as a result, the number of couplings (i.e.\ the number of nonzeros per row of the system matrix) will also scale like $\mathcal{O}(p^d)$.
Consequently, the system matrix will have $\mathcal{O}(p^{2d})$ nonzero entries, and thus the memory required to form the matrix system matrix grows quadratically with the number of degrees of freedom under $p$-refinement.
This is in contrast to $h$-refinement, in which the memory required for the system matrix grows only linearly with the number of degrees of freedom.
This motivates the use of matrix-free operator evaluation, whereby the action of the operator applied to a vector is computed without forming the system matrix.

In the matrix-free context, the memory requirements for the high-order operator can be reduced to $\mathcal{O}(p^d)$.
Naive implementations of the action of the operator will still require $\mathcal{O}(p^{2d})$ operations.
However, making use of sum factorization techniques, this complexity can be reduced to $\mathcal{O}(dp^{d+1})$ \cite{Orszag1980,Pazner2018}.
These techniques have been shown to be effective on modern many-core and accelerator-based architectures \cite{Brown2010,Muthing2017}.
In addition to efficient algorithms for the evaluation of the action of the operator, solving the resulting linear systems in a practical setting also requires effective preconditioners.
The development of preconditioners in the matrix-free setting is particularly challenging, because traditional matrix-based methods such as Gauss-Seidel or algebraic multigrid require access to the entries of the matrix \cite{Ruge1987}.
Furthermore, element-based domain decomposition methods such as additive and multiplicative Schwarz methods necessitate the solution of local problems, which, if performed by means of a direct solver, requires $\mathcal{O}(p^{3d})$ operations \cite{Toselli2005}.

There has been significant past work on the development of preconditioners suitable for use in the high-order matrix-free context.
Much of this work has made use of related low-order, sparse discretizations in order to construct preconditioners for the high-order system, an idea originally introduced by Orszag in 1980, and since built upon by numerous others \cite{Orszag1980,Deville1990,Canuto1985,Canuto2010}.
New low-order methods for preconditioning high-order spectral element discretizations were developed in \cite{BelloMaldonado2019}.
The automatic construction of sparse preconditioners for high-order finite element discretizations was considered in \cite{Austin2012}.
Matrix-free multigrid methods using polynomial smoothers (Cf.\ \cite{Adams2003}) were considered in \cite{Kronbichler2019}.
Matrix-free approximate block Jacobi methods using Kronecker product approximations were constructed for discontinuous Galerkin discretizations of conservation laws in \cite{Pazner2018} and extended to interior penalty discretizations in \cite{Pazner2018b}.
Overlapping Schwarz solvers for the spectral element discretization of the Navier-Stokes equations were proposed in \cite{Fischer1997}.
Hybrid multigrid solvers with Schwarz smoothers for high-order spectral element discretizations were developed in \cite{Lottes2005}, and extended to the Navier-Stokes equations in \cite{Fischer2005}.
These Schwarz smoothers are based on solving structured subdomain problems using the fast diagonalization method.

In this work, we construct preconditioners for high-order finite element discretizations based on corresponding low-order ($p=1$) finite element discretizations formed on a refined mesh.
The refined mesh is generated using a structured grid of Gauss-Lobatto points within each element.
These preconditioners make use of the spectral equivalence between the two discretizations, widely known as the finite element method--spectral element method (FEM-SEM) equivalence \cite{Canuto2007}.
The resulting low-order discretization is sparse, and, because the polynomial degree is fixed, its system matrix can be formed using only constant memory per degree of freedom.
In order to precondition the high-order system, it is necessary to solve the linear system corresponding to the low-order system.
This is challenging because the low-order refined mesh is anisotropic and not shape-regular with respect to $p$.
Thus, standard algebraic multigrid or geometric multigrid methods with pointwise smoothers do not converge uniformly with respect to $p$.
To address this issue, we develop ordered ILU smoothers that perform the function of line relaxation in the context of unstructured meshes.
Distinguishing this solver from other Schwarz-based methods, the use of ordered ILU smoothing allows the solver to handle domains with highly anisotropic meshes.
The method is parallelized using an additive Schwarz method based on overlapping subdomains.
The resulting preconditioner is robust in the polynomial degree $p$ and mesh size $h$.
The preconditioner can also be extended to discontinuous Galerkin methods, for which it is also robust in the size of the penalty parameter $\eta$.

The structure of the paper is as follows.
In Section \ref{sec:problem} we introduce the model problem and finite element discretization.
Then, in Section \ref{sec:matrix-free} we design and analyze the matrix-free preconditioners considered in this paper.
In Section \ref{sec:numerical}, we present a variety of numerical results using finite element and discontinuous Galerkin discretizations on structured and unstructured meshes.
We end with conclusions in Section \ref{sec:conclusions}.

\section{Model problem and discretization} \label{sec:problem}

In this work, we will consider the solution to the model Poisson problem with homogeneous Dirichlet boundary conditions,
\begin{equation} \label{eq:poisson}
   -\nabla \cdot (b \nabla u) = f \quad \text{in $\Omega$},
   \qquad u = 0 \quad \text{on $\partial\Omega$},
\end{equation}
where the coefficient $b$ and right-hand side $f$ are sufficiently smooth, and $b$ is bounded away from zero on the spatial domain $\Omega \subseteq \R^d$, $d \in \{2,3\}$.
To simplify the exposition, in what follows we will take $b \equiv 1$.
However, numerical examples with variable coefficients are considered in Section \ref{sec:numerical}.
The solution to this problem is approximated using a high-order finite element method.
To begin, we consider standard $H^1$ conforming discretizations.
Discontinuous Galerkin discretizations of this problem are presented in Section \ref{sec:dg}.
The spatial domain $\Omega \subseteq \R^d$ is tessellated using a mesh of tensor-product elements denoted $\T_p$.
An element $D \in \T_p$ is given as the image of the unit cube $[0,1]^d$ under a suitable transformation mapping.
We introduce the usual continuous piecewise polynomial finite element space of degree at most $p$ in each variable on the mesh $\T_p$, which we denote $V_p$.

The finite element problem corresponding to \eqref{eq:poisson} is to find $u_p \in V_p$ such that
\begin{equation} \label{eq:fem}
   ( \nabla u_p, \nabla v_p ) = (f, v_p), \qquad \text{for all $v_p \in V_p$,}
\end{equation}
where $(\cdot, \cdot)$ denotes the standard $L^2$ inner product on $\Omega$, defined by
\begin{equation}
   (u, v) = \int_\Omega u(\bm x) v(\bm x) \, d\bm{x}.
\end{equation}
Introducing a basis for the space $V_p$, we can write \eqref{eq:fem} as the algebraic system
\begin{equation} \label{eq:system}
   K_p \bm{u} = M_p \bm{f},
\end{equation}
where $K_p$ is the stiffness matrix corresponding to the bilinear form $a(\cdot, \cdot) = (\nabla\cdot,\nabla\cdot)$, $M_p$ is the mass matrix corresponding to the $L^2$ inner product $(\cdot,\cdot)$, and $\bm u$ and $\bm f$ are vectors of coefficients in the chosen basis.
It is well-known that the system of equations \eqref{eq:system} quickly becomes ill-conditioned when the polynomial degree $p$ is large or the mesh spacing $h$ is small.
For that reason, effective preconditioners are required for the efficient iterative solution of \eqref{eq:system}.
Additionally, the number of nonzeros in the stiffness and mass matrices scales like $\mathcal{O}(p^{2d})$.
Naive assembly of this matrix (using local dense matrix-matrix products) requires $\mathcal{O}(p^{3d})$ operations, and sum-factorized matrix assembly can be performed in $\mathcal{O}(p^{2d+1})$ operations \cite{Melenk2001}.
This computational cost, in terms of both memory and arithmetic operations, is seen to be prohibitive for high or even moderate $p$, especially in three dimensions.
On the other hand, the action of the operators $K_p$ and $M_p$ may be computed using techniques such as sum factorization with $\mathcal{O}(p^{d+1})$ operations and $\mathcal{O}(p^d)$ memory \cite{Orszag1980}.
The purpose of this paper is the development of \textit{matrix-free} preconditioners for \eqref{eq:system}, which can be constructed without explicit knowledge of the entries of the matrices $K_p$ and $M_p$, whose total memory cost scales like the number of degrees of freedom ($\mathcal{O}(p^d)$), and whose computational complexity is no more than that of applying the operator.

\subsection{Discontinuous Galerkin Methods} \label{sec:dg}

Also of interest are discontinuous Galerkin discretizations of \eqref{eq:poisson}.
Let $V_{\DG}$ denote the degree-$p$ discontinuous piecewise polynomial space defined on the mesh $\T_p$.
No continuity is enforced between the elements of the mesh.
There are a wide range of DG discretizations for elliptic problems \cite{Arnold2002}.
We consider the symmetric interior penalty (IP) discretization \cite{Arnold1982}, whose formulation is as follows: find $u_{\DG} \in V_{\DG}$ such that
\begin{equation} \label{eq:ip}
   a_{\IP}(u_{\DG}, v_{\DG}) =
   (\nabla u_{\DG}, \nabla v_{\DG})
      - \langle \{ \nabla u_{\DG} \}, \llb v_{\DG} \rrb \rangle
      - \langle \llb u_{\DG} \rrb, \{ \nabla v_{\DG} \} \rangle
      + \langle \sigma \llb u_{\DG} \rrb, \llb v_{\DG} \rrb \rangle
      = (f, v_{\DG}),
\end{equation}
for all $v_{\DG} \in V_{\DG}$.
Here, $\langle \cdot, \cdot \rangle$ denotes integration over the union of all faces of elements $D \in \T_p$, which we denote by $\Gamma$.
Consider a common face shared by two mesh elements, $D^-$ and $D^+$.
Then, $\{\cdot\}$ and $\llb \cdot \rrb$ denote the average and jump operators, respectively, defined by
\begin{equation}
   \{ \phi \} = \frac{1}{2}( \phi^- + \phi^+), \qquad
   \llb \phi \rrb = \phi^- \bm n^- + \phi^+ \bm n^+,
\end{equation}
where the superscripts $\pm$ denote the traces from within the elements $D^\pm$.
Similarly, $\bm n^\pm$ denotes the outward facing normal vector from within $D^\pm$.

The parameter $\sigma$ in \eqref{eq:ip} is a penalty parameter, which is required to be sufficiently large in order to stabilize the method.
In particular, $\sigma$ must be chosen to scale like $p^2/h$ to ensure that the system is positive-definite \cite{Arnold2002}, and so we write $\sigma = \eta p^2/h$.
However, large values of the penalty parameter result in ill-conditioned systems \cite{Shahbazi2005}.
In this paper, we will seek to design preconditioners for the DG system \eqref{eq:ip} whose convergence is independent of the value of the penalty parameter.

In addition to the symmetric interior penalty discretization \eqref{eq:ip}, we are also interested in alternative DG discretizations, such as the second method of Bassi and Rebay (BR2) \cite{Bassi1997b}.
This method makes use of a different stabilization scheme based on lifting operators.
In the BR2 method, the penalty term $\langle \sigma \llb u_{\DG} \rrb, \llb v_{\DG} \rrb \rangle$ in \eqref{eq:ip} is replaced by an alternative penalty term $\langle \alpha(u_{\DG}) , \llb v_{\DG} \rrb \rangle$.
The term $\alpha(u_{\DG})$ is defined on a face $e$ by $\alpha(u_{\DG}) = -\eta \{ r_e(\llb u_{\DG} \rrb) \}$, where the lifting operator $r_e$ is given by
\begin{equation}
   \int_\Omega r_e(\bm \varphi) \cdot \bm \tau \, dx
      = - \int_e \bm \varphi \cdot \{ \bm \tau \} \, ds.
\end{equation}
We can see that the stabilization term satisfies
$
\langle \alpha(u_{\DG}) , \llb v_{\DG} \rrb \rangle
   = \sum_e \left( \eta r_e(\llb u_{\DG} \rrb), r_e(\llb v_{\DG} \rrb) \right),
$
and thus the BR2 bilinear form $a_{\BR}(\cdot,\cdot)$ can be written
\begin{equation} \label{eq:br2}
   a_{\BR}(u_{\DG}, v_{\DG}) =
   (\nabla u_{\DG}, \nabla v_{\DG})
      - \langle \{ \nabla u_{\DG} \}, \llb v_{\DG} \rrb \rangle
      - \langle \llb u_{\DG} \rrb, \{ \nabla v_{\DG} \} \rangle
      + \sum_e \left(\eta r_e(\llb u_{\DG} \rrb), r_e(\llb v_{\DG} \rrb)\right).
\end{equation}
The BR2 method has the advantage that the coefficient $\eta$ of the stabilization term can be chosen to be $\mathcal{O}(1)$, and is not required to scale with the mesh spacing \cite{Brezzi2000}.
However, multigrid solvers often struggle with discretizations involving lifting operators \cite{Antonietti2015,Olson2011,Fortunato2019}.

\section{Matrix-free preconditioning} \label{sec:matrix-free}

We now describe the construction of a class of matrix-free preconditioners for the high-order Poisson problem \eqref{eq:system}.
These preconditioners are based on a structured, geometric multigrid V-cycle applied to a low-order refined discretization described in Section \ref{sec:fem-sem}.
Parallelization of the method is performed using an overlapping additive Schwarz method, described and analyzed in Section \ref{sec:as}.
The low-order formulation gives rise to anisotropy that is treated by means of line smoothing.
An algebraic alternative to line smoothing suitable for use on unstructured meshes is ordered incomplete LU (ILU), which is considered in Section \ref{sec:ilu}.
The extension to discontinuous Galerkin discretizations, also by means of an additive Schwarz method, is described in \ref{sec:dg-precond}.

\begin{rem}
Throughout what follows, we will use some notational conventions.
We will use $x \lesssim y$ and $x \gtrsim y$ to mean $x \leq C y$ and $x \geq C y$, respectively, where $C$ is a generic constant that is independent of the relevant discretization parameters (e.g.\ $h$ and $p$, which represent the mesh spacing and polynomial degree of the high-order problem).
When dealing with high-order objects (e.g.\ spaces, meshes, and operators) we will generally use a subscript $p$.
Likewise, a subscript $h$ will be used to indicate the corresponding low-order (i.e.\ multilinear) object.
\end{rem}

\subsection{Low-order refined equivalence} \label{sec:fem-sem}

The main tool used in the construction of our preconditioners is the spectral equivalence between the high-order finite element operator $K_p$ and a low-order finite element operator $K_h$ defined on the \textit{low-order refined} mesh.
Essential to the construction of the refined mesh will be a structured sub-grid of Gauss-Lobatto point within each element, which will give rise to the spectral equivalence between the high-order and low-order operators~\cite{Canuto2006,Canuto1982}.
This refined mesh, which, for simplicity of exposition, we assume in this section to be affine, is defined by subdividing each mesh element $D \in \T_p$ into mapped images of parallelepipeds with vertices at adjacent Gauss-Lobatto nodes in each dimension.
The resulting refined mesh is denoted $\T_h$.
The low-order stiffness matrix $K_h$ corresponds to a standard multilinear ($p=1$) finite-element discretization on the mesh $\T_h$.
This low-order refined stiffness matrix has two important properties that make it a suitable preconditioner in our setting.
First, $K_h$ is much sparser than $K_p$: the number of non-zeros per row of this matrix is bounded independently of the polynomial degree $p$ of the high-order operator.
Secondly, the low-order operator $K_h$ is spectrally equivalent to the high-order operator $K_p$, where the constants of equivalence are independent of $p$.
This equivalence is often referred to as FEM-SEM equivalence \cite{Chalmers2018,Canuto2007}.

Following the work of Canuto, Hussaini, Quarteroni, and Zang \cite{Canuto2007,Canuto2006} and Canuto, Gervasio, and Quarteroni \cite{Canuto2010}, we give a very brief overview of the FEM-SEM equivalence.
First, it is useful to establish the spectral equivalence in one spatial dimension between a Galerkin spectral method (``G-NI'') and piecewise linear finite elements on the Gauss-Lobatto points on the unit interval $[0,1]$.
Given a degree-$p$ polynomial $\phi_p$ defined on $[0,1]$, we associate to it its piecewise linear interpolant at the $p+1$ Gauss-Lobatto points, denoted $\phi_h$.
We then have the following result due to Canuto. \begin{prop}[{[\citenum{Canuto1994}, Propositions 2.1 and 2.2]}] There exist constants $c$ and $c'$ independent of $p$ such that
\[
   \frac{1}{c} \| \phi_p \|_{L^2([0,1])} \leq \| \phi_h \|_{L^2([0,1])} \leq
   c \| \phi_p \|_{L^2([0,1])}
\]
and
\[
   \frac{1}{c'} \| \phi_p' \|_{L^2([0,1])} \leq \| \phi_h' \|_{L^2([0,1])} \leq
   c' \| \phi_p' \|_{L^2([0,1])}.
\]
\end{prop}
This equivalence in $L^2$ norm and $H^1$ seminorm immediately gives spectral equivalence of the constant-coefficient one-dimensional spectral and finite element mass and stiffness matrices, which we denote here by $M^{\oned}_p, M^{\oned}_h, K^{\oned}_p,$ and $K^{\oned}_h$.
The extension to multiple space dimensions is straightforward when using tensor-product elements.
Define $K^j_p$ by
\[
   K^j_p = \bigotimes_{i=1}^d G_i^j, \qquad\text{where }
   G_i^j = \begin{cases}
      K^{\oned}_p &\quad \text{if $i=j$,} \\
      M^{\oned}_p &\quad \text{if $i\neq j$,}
   \end{cases}
\]
i.e.\ $K^j_p$ is given by the Kronecker product of $(d-1)$ one-dimensional mass matrices, and the one-dimensional stiffness matrix in the $j$th place.
We define $K^j_h$ analogously.
The mass and stiffness matrices on the $d$-dimensional unit cube $[0,1]^d$ are then given by
\begin{align*}
   M_p &= \bigotimes_{i=1}^d M^{\oned}_p, \\
   M_h &= \bigotimes_{i=1}^d M^{\oned}_h, \\
   K_p &= \sum_{j=1}^d K^j_p,\\
   K_h &= \sum_{j=1}^d K^j_h.
\end{align*}
Spectral equivalence of the $d$-dimensional operators follows from the product formula for eigenvalues of Kronecker products and the Courant-Fischer min-max theorem.
Finally, a simple Rayleigh quotient argument allows us to conclude the following spectral equivalence of the low-order and high-order operators.
\begin{prop}[Cf.\ section 6.3.4 of \cite{Canuto2007}] \label{prop:fem-sem}
There exist constants $c$ and $C$, independent of $h$ and $p$ such that
\[
   c \bm v^T K_h \bm v \leq \bm v^T K_p \bm v \leq C \bm v^T K_h \bm v,
\]
and
\[
   c \bm v^T M_h \bm v \leq \bm v^T M_p \bm v \leq C \bm v^T M_h \bm v,
\]
where $M_p$ and $K_p$ are the high-order finite element operators from \eqref{eq:system} defined on the mesh $\T_p$, and $M_h$ and $K_h$ are the low-order operators defined on the Gauss-Lobatto low-order refined mesh $\T_h$.
\end{prop}

As a consequence of this result, a good preconditioner for the low-order refined system $K_h$ will also be a good preconditioner for the original high-order system $K_p$.
Thus, we devote our attention to the construction of robust and scalable preconditioners for the systems $K_h$, which are obtained through a low-order refinement procedure.

\begin{rem}
The main challenge in designing a preconditioner for the low-order refined operator $K_h$ is that the resulting mesh is anisotropic.
The spacing of Gauss-Lobatto points near the interval endpoints scales like $1/p^2$, and the spacing of Gauss-Lobatto points near the midpoint of the interval scales like $1/p$ \cite{Brix2014}.
Thus, the aspect ratio of the parallelepipeds making up the Cartesian grid generated from the one-dimensional Gauss-Lobatto points scales like $p$.
As a consequence, the low-order refined meshes $\T_h$ are not shape regular in $p$, and we can expect degraded convergence of multigrid-type algorithms \cite{Hackbusch1985}.
\end{rem}

\subsection{Additive Schwarz and parallel subspace correction methods}
\label{sec:as}

In order to reduce the global problem $K_h$ into a set of smaller, local problems that can be solved in parallel, we make use of the additive Schwarz domain decomposition framework \cite{Toselli2005,Dryja1990,Pavarino1993}.
The main idea of the additive Schwarz method is to decompose the finite element space $V_h$ into a sum of subspaces
\begin{equation}
   V_h = V_0 + V_1 + \cdots + V_J.
\end{equation}
At this point, it will be useful to introduce some notation.
We define the operator $A_h : V_h \to V_h$ by $(A_h u_h, v_h) = a(u_h,v_h)$ where $(\cdot, \cdot)$ is the $L^2$ inner product, and $a(u_h, v_h) = (\nabla u_h, \nabla v_h)$.
Furthermore, for each subspace $V_j$, we define the elliptic projection $P_j : V_h \to V_j$, satisfying
\begin{equation} \label{eq:elliptic-proj}
   a(P_j u_h, v_j) = a(u_h, v_j) \qquad \text{for all $v_j \in V_j$},
\end{equation}
and the $L^2$ projection $Q_j : V_h \to V_j$ satisfying
\begin{equation} \label{eq:l2-proj}
   (Q_j u_h, v_j) = (u_h, v_j) \qquad \text{for all $v_j \in V_j$}.
\end{equation}
On each subspace $V_j$, let $A_j$ be the restriction of $A_h$ to $V_j$.
Then, we have the following useful identity,
\begin{equation} \label{eq:identity}
   A_j P_j = Q_j A_h,
\end{equation}
and so $P_j = A_j^{-1} Q_j A_h$.
However, it is usually not feasible to invert the subspace operators $A_j$ exactly, and thus we introduce approximate local solvers $R_j \approx A_j^{-1}$.
The additive Schwarz preconditioner is then given by
\begin{equation} \label{eq:as}
   B = \sum_{j=0}^J R_j Q_j.
\end{equation}
Each of the local approximate solvers $R_j$ may be applied in parallel.
Thus, this method is also often referred to as the method of parallel subspace corrections \cite{Xu1992}.

In what follows, we will be interested in studying the convergence properties of a particular choice of additive Schwarz method.
The main quantity we are interested in studying is the \textit{iterative condition number} of the preconditioned system.
For a generic operator $A$, which we assume is similar to a symmetric positive-definite matrix, the iterative condition number is defined as the ratio of largest eigenvalue to the smallest eigenvalue,
\begin{equation}
   \kappa(A) = \frac{\lambda_{\rm max}(A)}{\lambda_{\rm min}(A)}.
\end{equation}
This quantity will control the speed of convergence of the preconditioned system when used with a Krylov subspace method such as conjugate gradients.

We now define the space decomposition used for the present solver.
The space $V_0$ is taken to be a coarse subspace, defined as the multilinear ($p=1$) finite element space on the original (coarse) mesh $\T_p$.
$V_0$ consists of piecewise polynomials of degree $p=1$, whose degrees of freedom are defined at the (low-order) vertices of the coarse mesh $\T_p$.
Note that the size of $V_0$ is independent of the polynomial degree $p$ of the high-order space.
Since we are interested in the case of high polynomial degree $p$, the mesh $\T_p$ is typically fairly coarse, and thus the space $V_0$ is only a small fraction of the size of the space $V_h$ (which is refined according to the degree $p$).
However, this space will usually still be too large for use with direct solvers.
In this work, we choose the coarse solver $R_0$ to be given by one V-cycle of the algebraic multigrid method BoomerAMG implemented in \textit{hypre} \cite{Falgout2002,Henson2002}.
In the massively parallel setting, the choice of coarse solver can be quite important.
We choose to use \textit{hypre} for this component of the solver because of its demonstrated scalability.
In principle, any scalable solver suitable for standard low-order finite element discretizations on unstructured meshes can be used for $R_0$.
Our experience suggests that for relatively small problems with coarse meshes at high order the convergence of the preconditioner $B$ is often not overly sensitive to the choice of coarse solver $R_0$, and in this case the cost of applying $R_0$ represents a relatively small portion of the overall algorithm.
For the remainder of this section, we will assume that $R_0$ is a uniform preconditioner for $A_0$, i.e.\ that $\kappa(R_0 A_0)$ is bounded independent of mesh size $h$.

The spaces $V_1, \ldots, V_J$ are defined in terms of overlapping, unstructured \textit{vertex patches} (cf.~\cite{Pavarino1993,Dryja1990}).
These patches are defined in terms of sets of vertices of the original coarse mesh $\T_p$.
We begin by partitioning the vertices of $\T_p$ into $J$ disjoint sets, $E_1, \ldots, E_J$.
To each set of vertices $E_j$, we associate a subdomain $\Omega_j \subseteq \Omega$, which is obtained by taking the union of all coarse elements $D \in \T_p$ containing any vertex $\bm x_k \in E_j$.
In other words, $\Omega_j$ is defined by
\begin{equation} \label{eq:patch}
   \Omega_j = \bigcup_{\bm x_k \in E_j}
    \bigcup_{\substack{D \in \T_p \\ \bm x_k \in D}} D.
\end{equation}
Note that in general, the subdomains $\Omega_j$ overlap, since each element $D \in \T_p$ has $2^d$ vertices, each of which may belong to a different vertex set $E_j$.
The number of overlapping subdomains containing a single element is bounded by the number of vertices per element.
The subspace $V_j$, for $j \geq 1$ is defined as the multilinear finite element space on the mesh $\T_h$ restricted to the subdomain $\Omega_j$.
Homogeneous Dirichlet conditions are enforced at the subdomain boundary $\partial\Omega_j$.
The linear systems $A_j$ associated with each patch subspace $V_j$ may be quite large, and so the approximate solvers $R_j$ associated with each of these subspaces are given by a structured geometric multigrid V-cycle with ordered ILU smoothing.
These approximate solvers will be described and studied in greater detail in Section \ref{sec:multigrid}.
For now, we will assume that $R_j$ is a uniform preconditioner for $A_j$, i.e.\ that $\kappa(R_j A_j)$ is bounded independent of mesh size $h$ and polynomial degree $p$.

We are now interested in studying the speed of convergence of the preconditioned system $B A_h$, where the additive Schwarz preconditioner $B$ is defined by \eqref{eq:as}.
To estimate the condition number of the preconditioned system, we largely follow the domain decomposition theory developed in~\cite{Dryja1990,Toselli2005,Xu1992,Xu2002}.
Note that the condition number $\kappa(BA_h)$ is bounded by $c_1 / c_0$, where $c_1$ and $c_0$ are constants satisfying
\begin{equation}
   c_0 a(u_h, u_h) \leq a(BA_h u_h, u_h) \leq c_1 a(u_h, u_h).
\end{equation}
To begin, we first make use the assumption that each of the subspace solvers gives a uniformly well-conditioned system.
In other words, that there exist constants $C_0$ and $C_1$, independent of $h$ and $p$, such that (for all $0 \leq j \leq J$)
\begin{equation} \label{eq:approx-solvers}
   C_0 a(u_j, u_j) \leq a(R_j A_j u_j, u_j) \leq C_1 a(u_j, u_j)
      \quad\text{for all $u_j \in V_j$}.
\end{equation}
Using this assumption and making use of the identity \eqref{eq:identity}, we
have
\begin{align*}
   a(BA_h u_h, u_h) = \sum_{j=0}^J a(R_j Q_j A_h u_h, u_h)
                    = \sum_{j=0}^J a(R_j A_j P_j u_h, u_h)
                    = \sum_{j=0}^J a(R_j A_j P_j u_h, P_j u_h),
\end{align*}
and by \eqref{eq:approx-solvers} we see
\begin{equation}
    C_0 \sum_{j=0}^J a(P_j u_h, u_h)
      \leq a(B A_h u_h, u_h)
      \leq C_1 \sum_{j=0}^J a(P_j u_h, u_h).
\end{equation}
From this, we conclude that it suffices to find bounds $\tilde{c}_0$ and
$\tilde{c}_1$ such that
\begin{equation} \label{eq:c-tilde}
   \tilde{c}_0 a(u_h, u_h) \leq \sum_{j=0}^J a(P_j u_h, u_h)
      \leq \tilde{c}_1 a(u_h,u_h),
\end{equation}
from which we immediately have $c_0 \geq \tilde{c}_0 C_0$ and $c_1 \leq
\tilde{c}_1 C_1$.

For the upper bound, we make a standard argument using the finite overlap of the mesh patches.
Note that, since the projection operators $P_j$ have norm one,
\begin{equation} \label{eq:finite-overlap}
   \sum_{j=0}^J a(P_j u_h, u_h)
      = \sum_{j=0}^J a(P_j u_h, P_j u_h)
      = \sum_{j=0}^J a_j(P_j u_h, P_j u_h)
      \leq \sum_{j=0}^J a_j(u_h, u_h),
\end{equation}
Each element is contained in a bounded number of overlapping patches (bounded by the number of vertices per element, $2^d$).
Accounting for the course space $V_0$, we then have $\tilde{c}_1 \leq 2^d + 1$.

It remains to estimate the lower bound $\tilde{c}_0$, for which we recall Lions' lemma \cite{Lions1988}.
\begin{lem}[Lions' lemma]
Let $u_h \in V_h$ be written as an element of $V_0 + V_1 + \cdots V_J$, $u_h = \sum_{j=0}^J u_j$, where $u_j \in V_j$.
If
\begin{equation} \label{eq:decomp}
   \sum_{j=0}^J a(u_j, u_j) \leq \tilde{c}_0^{-1} a(u_h, u_h)
\end{equation}
then
\[
   \tilde{c}_0 a(u_h, u_h) \leq \sum_{j=0}^J a(P_j u_h, u_h).
\]
\end{lem}
By means of this lemma, it suffices to demonstrate a stable decomposition $u_h = \sum_{j=0}^J u_j$ for $u_h \in V_h$ satisfying inequality \eqref{eq:decomp}.
On the coarse space, we use the $L^2$ projection to define $u_0 = Q_0 u_h$.
Making use of the results from Bramble and Xu \cite{Bramble1991}, we have the following properties:
\begin{equation} \label{eq:l2-properties}
   a(Q_0 u_h, Q_0 u_h) \lesssim a(u_h, u_h), \quad\text{and}\quad
   \| u_h - Q_0 u_h \|^2_{L^2} \lesssim h^2 a(u_h, u_h).
\end{equation}
Now it remains to write $w = \sum_{j=1}^J u_j$, where $w = u_h - u_0$.

We begin by defining a partition of unity $\{\theta_j\}$ subordinate to the patches $\Omega_j$.
Each function of the partition of unity will be a member of the low-order coarse finite element space $V_0$.
Recall that the space $V_0$ is given by the $p=1$ finite element space defined on the coarse mesh $\T_p$, and therefore $\theta_j : \Omega \to \mathbb{R}$ is completely determined by its values at the coarse mesh vertices $\bm x_k$.
Additionally, recall that the subdomains $\Omega_j$ are defined in terms of sets of vertices $E_j$.
Then, let $\theta_j \in V_0$ be defined by
\begin{equation} \label{eq:partn-unity}
   \theta_j(\bm x_k) = \begin{cases}
      1 & \quad\text{if $\bm x_k \in E_j$},\\
      0 & \quad\text{if $\bm x_k \notin E_j$}.
   \end{cases}
\end{equation}
It is clear from this definition that $\mathrm{supp}(\theta_j) \subseteq \Omega_j$.
Since each coarse mesh vertex $\bm x_k$ is included in exactly one vertex set $E_j$, we have that $\sum_{j=1}^J \theta_j(\bm x) = 1$ for all $\bm x \in \Omega$.
Letting $w_j = \theta_j w$, we have $\sum_{j=1}^J w_j = w$.
However, note that in general, $w_j \notin V_j$, since $w_j$ may include quadratic terms.
Therefore, define $u_j = I_h(w_j) \in V_j$, where $I_h$ denotes multilinear interpolation on the refined mesh $\T_h$.
By linearity of $I_h$, the functions $u_j$ also satisfy $\sum_{j=1}^J u_j = w$ and thus $u_0 + \sum_{j=1}^J u_j = u_h$.

We make the following claim concerning the $H^1$ seminorm of $u_j$.
This lemma is the key technical result allowing us to prove uniformity of the preconditioner with respect to $p$ as well as $h$.
We thank D.\ Kalchev for his contributions to the proof of this lemma.
\begin{lem} \label{lem:uj}
   Let $w \in V_h$ be given, and let $w_j = \theta_j w$, where $\theta_j$ is
   defined by \eqref{eq:partn-unity}. Let $u_j = I_h(w_j)$. Then,
   \[
      | u_j |_{H^1(\Omega_j)}^2
      \lesssim h^{-2} \| w \|^2_{L^2(\Omega_j)}
                          + | w |^2_{H^1(\Omega_j)}.
   \]
\end{lem}
\begin{rem}
The main issue that must be addressed in this lemma is that the mesh $\T_h$ is not shape-regular with respect to $p$, even though we assume that the original coarse mesh $\T_p$ is shape regular.
The aspect ratio of elements in $\T_h$ increases linearly with the polynomial degree $p$ of the high-order space $V_p$.
For this reason, standard inverse inequalities (such as those used in e.g.\ \cite{Dryja1990}) are insufficient to prove this claim.
\end{rem}
\begin{proof}
We first claim that
\begin{equation} \label{eq:wi-bound}
   | w_j |_{H^1(\Omega_j)}^2 \lesssim h^{-2} \| w \|_{L^2(\Omega_j)}^2
      + | w |_{H^1(\Omega_j)}^2 .
\end{equation}
We work on a single element in the refined mesh $D \in \T_h$.
On $D$, both $w$ and $\theta_j$ are bilinear functions (and hence smooth).
So, a simple application of the product rule gives
\begin{align*}
   | \nabla(\theta_j w) |^2
      = | (\nabla\theta_j) w + \theta_j \nabla w |^2
      \leq 2 | (\nabla\theta_j) w |^2 + 2 | \theta_j \nabla w |^2.
\end{align*}
Given the definition of $\theta_j$, we have the pointwise
\[
   0 \leq \theta_j \leq 1, \qquad | \nabla \theta_j | \lesssim 1/h.
\]
Making use of these bounds, we obtain
\begin{align*}
   | w_j |_{H^1(D)}^2 \lesssim  h^{-2} \| w \|_{L^2(D)}^2
      + | w |_{H^1(D)}^2 .
\end{align*}
Thus, summing over all elements $D \subseteq \Omega_j$, we arrive at \eqref{eq:wi-bound}.

Now, it remains to show that $| I_h(w_j) |_{H^1(D)}^2 \lesssim | w_j |_{H^1(D)}^2$, where the constant is independent of the aspect ratio of the element $D$.
It suffices to consider the case that $D$ is the rectangle $[0,h_1] \times\cdots \times[0,h_d]$.
This simplification can be performed because the coarse mesh element containing $D$ can be mapped to the reference element $[0,1]^d$, and the resulting $H^1$ seminorms will be equivalent up to a constant depending on the regularity of the coarse mesh $\T_p$.

We proceed dimension-by-dimension.
Define the $i$th-directional seminorm $|\cdot|_{H^1_i(D)}$ by
\[
   |u|_{H^1_i(D)}^2 =
      \int_D \left|\frac{\partial u}{\partial x_i}\right|^2\,d\bm x.
\]
From this definition, it follows that
\[
   | u |_{H^1(D)}^2 = \sum_{i=1}^d | u |_{H^1_i(D)}^2.
\]
Define $\hat{u}_j$ on the reference element $[0,1]^d$ by $\hat{u}_j(\bm x) = u_j(\bm h \bm x)$, where $\bm h \bm x = (h_1 x_1, \ldots, h_d, x_d)$.
Then, a simple change of variables shows that
\[
   | u_j |_{H^1_i(K)}^2 = \frac{\prod_{j=1}^d h_j}{h_i^2}
      | \hat{u}_j |_{H^1_i([0,1]^d)}^2.
\]
Note that since $u_j = I_h(w_j)$, we have $\hat{u}_j = \widehat{I_h(w_j)} = I_h(\hat{w}_j)$.
Furthermore, let $\Q_2([0,1]^d)$ be the space of polynomials on the unit cube that are at most quadratic in each variable, and let $\Q_1([0,1]^d)$ be the space of multilinear polynomials on the unit cube.
Then, since this space is finite-dimensional, it is clear that $I_h : \Q_2([0,1]^d) \to \Q_1([0,1]^d)$ is a bounded linear operator with the respect to the $|\cdot|_{H^1_i(D)}$ seminorms.
In other words, there exists some constant $c_I$ such that, for $1 \leq i \leq d$,
\begin{equation} \label{eq:interp-bound}
   | I_h(\hat{w}_j) |^2_{H^1_i([0,1]^d)}
      \leq c_I | \hat{w}_j |^2_{H^1_i([0,1]^d)}.
\end{equation}
Therefore, we have
\[
   | u_j |_{H^1(D)}^2
       = \sum_{i=1}^d | u_j |_{H^1_i(D)}^2
       = \sum_{i=1}^d \frac{\prod_{j=1}^d h_j}{h_i^2}
         |\hat{u}_j|_{H^1_i([0,1]^d)}^2
       = \sum_{i=1}^d \frac{\prod_{j=1}^d h_j}{h_i^2}
         | I_h(\hat{w}_j) |_{H^1_i([0,1]^d)}^2.
\]
Making use of the inequality \eqref{eq:interp-bound}, this gives
\begin{equation} \label{eq:aniso-interp-bound}
   | u_j |_{H^1(D)}^2
      \leq c_I \sum_{i=1}^d \frac{\prod_{j=1}^d h_j}{h_i^2}
                 | \hat{w}_j |_{H^1_i([0,1]^d)}^2
      = c_I \sum_{i=1}^d | w_j |_{H^1_i(D)}^2
      = c_I | w_j |_{H^1(D)}^2,
\end{equation}
where it is important to note that the constant $c_I$ is independent of the aspect ratio of the rectangle $[0,h_1] \times \cdots \times [0,h_d]$, and hence independent of the high-order degree $p$.
Summing over all elements $D \subseteq \Omega_j, D \in \T_h$, we combine \eqref{eq:aniso-interp-bound} with \eqref{eq:wi-bound} to conclude
\[
   | u_j |_{H^1(\Omega_j)}^2
   \lesssim h^{-2} \| w \|^2_{L^2(\Omega_j)}
                       + | w |^2_{H^1(\Omega_j)}. \qedhere
\]
\end{proof}

\begin{rem} \label{rem:anisotropy}
Note that the above lemma additionally applies to mesh patches $\Omega_j$ featuring anisotropy from sources other than Gauss-Lobatto refinement, suggesting that this additive Schwarz technique could prove to be useful when applied to problems with boundary layers or other more general anisotropic features.
An example of such a problem is considered in Section \ref{sec:aniso}.
\end{rem}

We are now ready to state the main result concerning the additive Schwarz preconditioner applied to the low-order refined system.
\begin{thm} \label{thm:as}
Define the preconditioner $B$ by \eqref{eq:as}, where each local solver $R_j$ is an approximation to $A_j^{-1}$ that is uniform in $h$ and $p$.
Then, there exists a constant $\tilde{C}$, independent of $h$ and $p$, such that
\[
   \kappa(BA_h) \leq \tilde{C}.
\]
\end{thm}
\begin{proof}
The constant $\tilde{C}$ is bounded above by $\tilde{c}_1/\tilde{c}_0$ from \eqref{eq:c-tilde}.
Using the finite overlap of the mesh patches, the estimate given by \eqref{eq:finite-overlap} implies that $\tilde{c}_1 \leq 2^d + 1$.

To estimate the lower-bound $\tilde{c_0}$ we use Lions' lemma and the decomposition $u_h = \sum_{j=0}^J u_j$ described above.
Property \eqref{eq:l2-properties} implies that $a(u_0, u_0) \lesssim a(u_h, u_h)$ and $\| w \|^2_{L^2} \lesssim h^{2} a(u_h,u_h)$.
Combining these estimates with the result of Lemma \ref{lem:uj}, we have
\begin{align*}
   \sum_{j=0}^J a(u_j, u_j)
      &\lesssim a(u_h, u_h) + \sum_{j=1}^J a(u_j, u_j) \\
      &\lesssim a(u_h, u_h) + \sum_{j=1}^J \left(
         h^{-2} \| w \|_{L^2(\Omega_j)}^2 + |w|^2_{H^1(\Omega_j)}
      \right) \\
      &\lesssim a(u_h, u_h) + \sum_{j=1}^J | u_h |^2_{H^1(\Omega_j)} \\
      &\lesssim (2^d + 1) a(u_h, u_h).
\end{align*}
Thus, $\tilde{c}_0 \gtrsim 1 / (2^d+1)$, completing the proof.
\end{proof}
Combining the estimates from Theorem \ref{thm:as} and Proposition \ref{prop:fem-sem}, this result extends to the high-order system \eqref{eq:system}.
\begin{cor} \label{cor:as}
   Let $B$ be the preconditioner defined by \eqref{eq:as}.
   Then, there exists a constant $C$, independent of $h$ and $p$, such that
   \[
      \kappa(B A_p) \leq C.
   \]
\end{cor}

\begin{rem}
The components described above can also be used to construct a simple two-grid scheme, using $V_1 + \cdots + V_J$ as a smoother and $V_0$ as a coarse solver.
Furthermore, the multiplicative version of the Schwarz solver is easily constructed.
In the remainder of this paper, we consider only the additive version of the algorithm for reasons of simplicity and ease of parallelization.
\end{rem}

\subsection{Element-structured multigrid V-cycle} \label{sec:multigrid}

In order to precondition the local patch operators $A_j$, $1 \leq j \leq J$, we introduce an element-structured multigrid V-cycle.
This V-cycle is used as a local approximate solver in an additive Schwarz framework, which in turn is used as a preconditioner in a Krylov subspace method.
Recall that each element $D \subseteq \Omega_j, D \in \T_p$ from the original high-order mesh is refined by meshing the Gauss-Lobatto nodes, resulting in a (non-uniform) Cartesian grid of sub-elements on each coarse element.
This gives rise to a natural geometric hierarchy of meshes, which can be obtained from the refined mesh by coarsening the Cartesian subgrid within each coarse element.

Each level in the hierarchy is constructed from the previous level by removing half of the interior Gauss-Lobatto points within each element, resulting in $\mathcal{O}(\log(p))$ levels.
For instance, this can be achieved by deleting every other interior one-dimensional Gauss-Lobatto point (proceeding, for example, left-to-right) within the coarse element $D$.
This process terminates when all interior Gauss-Lobatto points have been removed.
At this point, the coarsened elements coincide with the elements from the original mesh $\T_p$.
This multigrid procedure bears some resemblance to $p$-multigrid methods, with $p=1$ as the coarsest level \cite{Helenbrook2003,Helenbrook2006,Fidkowski2005}.
The key difference is the current method operates only on low-order refined operators, avoiding the prohibitive cost of assembling the stiffness matrices associated with high-order operators.
At each level of the multigrid hierarchy, we perform one pre-smoothing and one post-smoothing step.
The construction of appropriate smoothers is discussed in Section \ref{sec:ilu}.
Depending on the size of the mesh patch, at the coarsest level we can use either a direct solver, or a less expensive approximate solver such as ILU.

\subsection{Ordered ILU smoothers} \label{sec:ilu}

\begin{figure} \centering
   \includegraphics[scale=0.82]{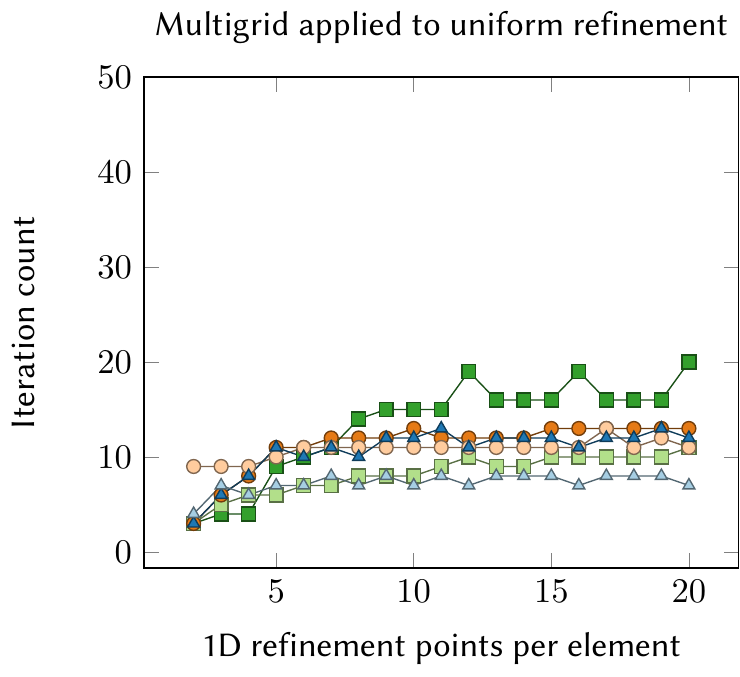}%
   \includegraphics[scale=0.82]{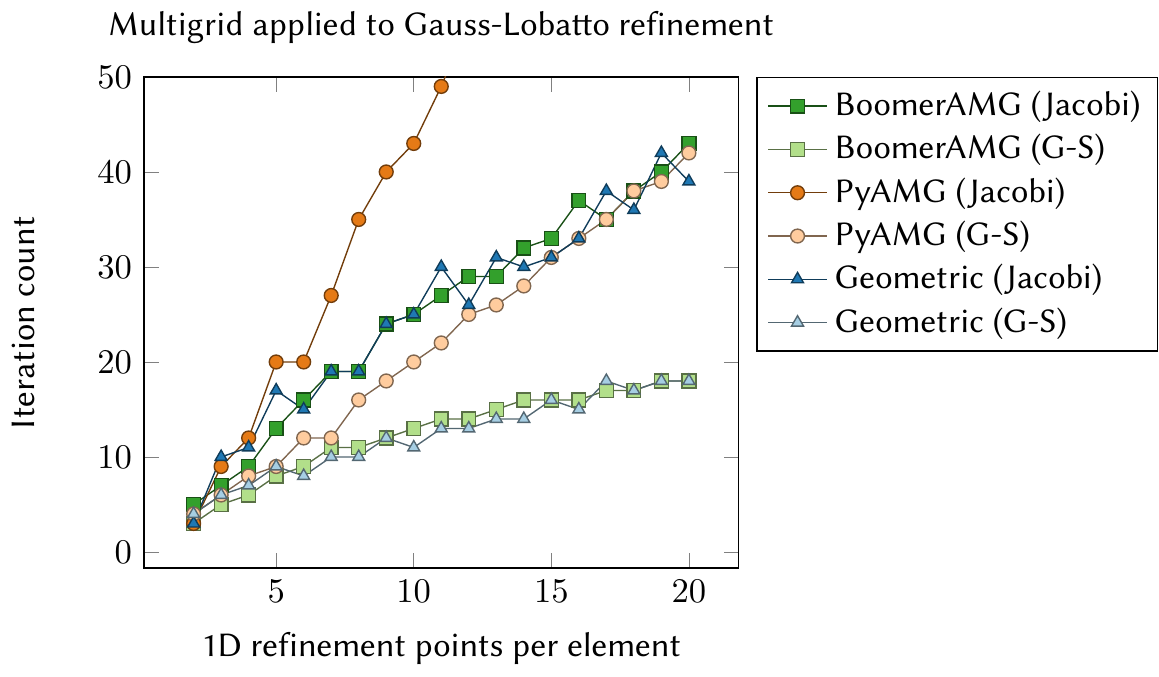}
   \caption{Multigrid-preconditioned conjugate gradient iteration counts (residual reduction of $10^{12}$) for a Poisson problem on a Cartesian grid. The grids are obtained by refining a $2\times2$ Cartesian grid using either uniform or Gauss-Lobatto refinement.}
   \label{fig:mg-iters}
\end{figure}

\begin{figure} \centering
   \includegraphics[scale=0.82]{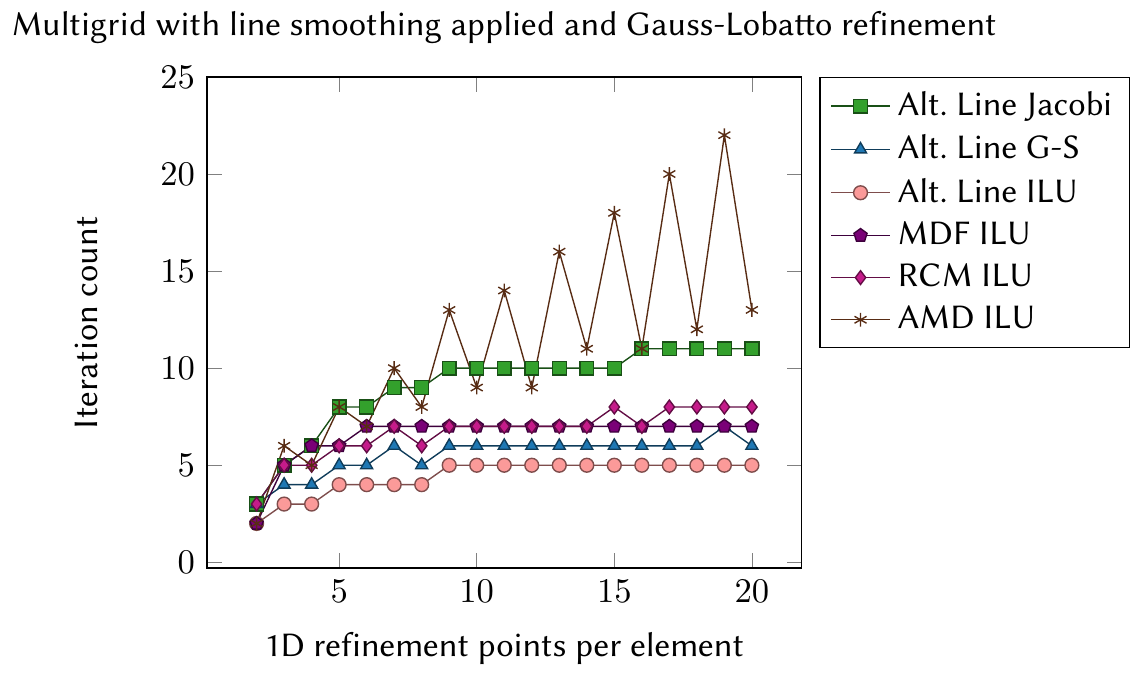}%
   \caption{Line and ILU smoothing: multigrid-preconditioned conjugate gradient iteration counts (residual reduction of $10^{12}$) for a Poisson problem on a Cartesian grid. The grids are obtained by refining a $2\times2$ Cartesian grid using Gauss-Lobatto refinement.}

   \label{fig:line-iters}
\end{figure}

As mentioned previously, the refined mesh $\T_h$ is highly anisotropic.
It is well-known that the performance of standard smoothers such as Jacobi and Gauss-Seidel depend on the anisotropy of the mesh, and thus degrade as the polynomial degree $p$ of the high-order system is increased.
To illustrate this, we apply a variety of algebraic and geometric multigrid methods to the Poisson problem on a $2\times2$ coarse mesh.
Each coarse element is subdivided using a regular grid defined by a given number of refinement points in each direction.
These points are either distributed uniformly or according to the Gauss-Lobatto rule.
In Figure \ref{fig:mg-iters}, we present the number of CG iterations required to reduce the residual by a factor of $10^{12}$ while increasing the number of refinement points per direction in each element.
We apply the commonly used BoomerAMG \cite{Henson2002} and PyAMG \cite{Olson2018} algebraic multigrid preconditioners, as well as a standard geometric multigrid V-cycle based on block-structured mesh coarsening.
For each of these multigrid methods, we consider both Jacobian and Gauss-Seidel smoothing.
When the grid is refined uniformly, geometric and algebraic multigrid method converge uniformly using the simple point Jacobi smoother.
However, when the grid is refined anisotropically using Gauss-Lobatto points, the iteration counts quickly grow.
Using a more powerful smoother such as Gauss-Seidel can improve performance, however the iteration counts are no longer robust in the refinement level.

A common method for addressing this difficulty is line smoothing \cite{Wesseling1995,Shen2000}.
On general, unstructured meshes, properly identifying nonintersecting lines of degrees of freedom along which to perform the smoothing is no longer trivial \cite{Hassan1993}.
To avoid this difficulty, we make use of incomplete LU (ILU) smoothing.
Different orderings of the degrees of freedom result in different ILU factorizations.
The choice of ordering is known to have a large effect on the convergence of ILU-preconditioned conjugate gradient methods \cite{DAzevedo1992}.
To make the connection between ILU smoothing and alternating line smoothing, we see that if the ordering is chosen such that the degrees of freedom on each line appear consecutively, then alternating applications of the ILU factorization can be used to perform the alternating line smoothing \cite{Heinrichs1988}.
We refer to this smoother as ``alternating line-ordered ILU.'' This procedure still requires the identification of lines corresponding to strong coupling.
We are interested in ordering the degrees of freedom in the context of unstructured meshes, and thus turn to algebraic ordering methods.
In particular, we consider approximate minimum degree (AMD) \cite{Amestoy1996}, reverse Cuthill-McKee (RCM) \cite{Cuthill1969}, and minimum discarded fill (MDF) \cite{DAzevedo1992,Persson2008} orderings.

We apply alternating line Jacobi, alternating line Gauss-Seidel, and ordered ILU smoothing to the model problem with Gauss-Lobatto refinement.
In Figure \ref{fig:line-iters} we show the number of CG iterations required to reduce the residual by a factor of $10^{12}$.
We first note that the alternating line relaxation methods all result in iteration counts that remain constant, independent of the refinement level.
Line Gauss-Seidel smoothing results in iteration counts that are about one-half of those obtained using line Jacobi smoothing.
Alternating line-ordered ILU results in iteration counts slightly less than alternating line Gauss-Seidel.
MDF-ordered and RCM-ordered ILU give comparable iteration counts, and both remain robust in the number of refinements.
AMD-ordered ILU results in iteration counts that exhibit an odd-even effect and that grow with the number of refinements, indicating that the AMD ordering is not suitable for our ILU smoothers.
The minimum discarded fill and reverse Cuthill-McKee orderings perform similarly for this simple test problem.
In the remainder of this work, we will make use of the MDF ordering because it has been shown to perform well on problems with high anisotropy and complex coefficients \cite{DAzevedo1992b}.
In Figure \ref{fig:smoothing}, we illustrate the smoothing properties of the above methods by displaying the error after two multigrid iterations, with a random initial guess and zero right-hand side.
It is clear that the pointwise smoothers such as Jacobi and Gauss-Seidel do not smooth the error in regions of anisotropy (i.e.\ at the medians and towards the edges of the domain).
Line Jacobi and AMD-ordered ILU smoothing perform better, but still suffer in anisotropic regions. Line Gauss-Seidel and the remaining ILU methods all result in smooth errors, even in regions of high anisotropy.

\begin{figure}
   \includegraphics[width=\linewidth]{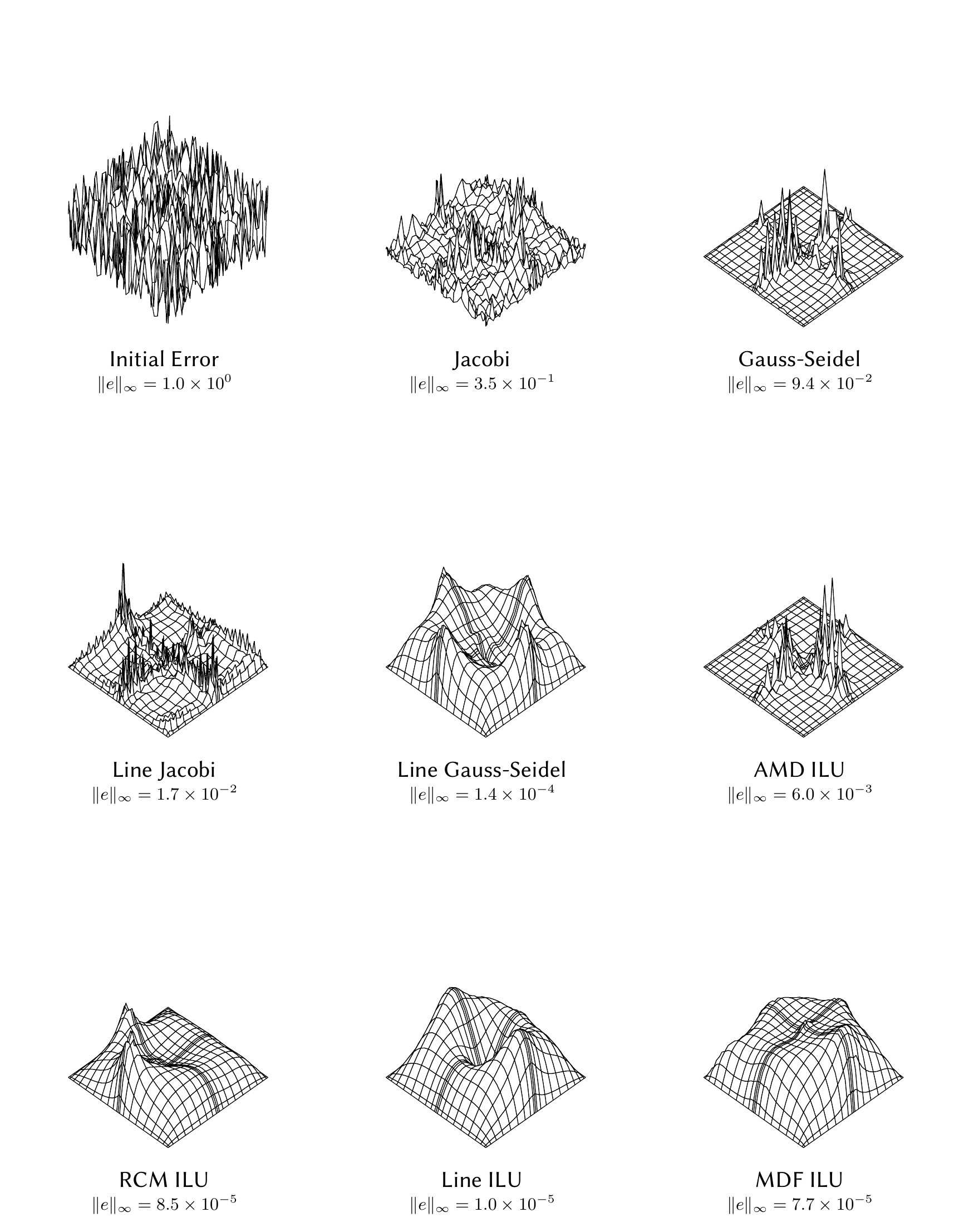}%
   \caption{Two multigrid iterations with different smoothers applied to
            $2\times2$ Cartesian grid with Gauss-Lobatto refinement and random
            initial error.}
   \label{fig:smoothing}
\end{figure}

\subsection{Extension to DG methods} \label{sec:dg-precond}

In this section, we consider the extension of the above preconditioners to discontinuous Galerkin (DG) discretizations of \eqref{eq:poisson}.
Of particular interest are the interior penalty and BR2 discretizations, given by \eqref{eq:ip} and \eqref{eq:br2}, respectively.
The goal is to obtain preconditioners whose convergence is independent of the DG penalty parameter $\eta$.
This is achieved very naturally in the additive Schwarz framework by making use of the space decomposition introduced in \cite{Antonietti2016}.
Recall that $V_{\DG}$ is the space of degree-$p$ piecewise polynomial space defined on the mesh $\T_p$, where no continuity is enforced between elements.
We write
\begin{equation} \label{eq:dg-decomp}
   V_{\DG} = V_B + V_p,
\end{equation}
where $V_p$ is the standard $H^1$ conforming degree-$p$ finite element space.
The space $V_B$ consists of all functions that vanish at all interior Gauss-Lobatto nodes.
Using the space decomposition \eqref{eq:dg-decomp}, the additive Schwarz preconditioner for the DG discretization becomes
\begin{equation}\label{eq:dg-as}
   B_{\DG} = R_B Q_B + R_p Q_p,
\end{equation}
where $R_B$ and $R_p$ are approximate solvers, $R_B \approx A_B^{-1}$ and $R_p \approx A_p^{-1}$.
We take $R_p$ to be the additive Schwarz preconditioner $R_p = B$ for the conforming problem, defined by \eqref{eq:as}, which, according to Theorem \ref{cor:as} is a uniform preconditioner for $A_p$.
As in \cite{Antonietti2016}, we take $R_B$ to be a simple point Jacobi preconditioner.
Then, we have the following result.
\begin{thm}[{[\citenum{Antonietti2016}, Theorem 2]}]
   Let $A_{\IP}$ denote the operator corresponding to the interior penalty bilinear form $a_{\IP}$.
   Let $R_B$ denote the point Jacobi preconditioner on $V_B$ and let $R_p$ be a uniform (in $h$ and $p$) approximation to $A_p^{-1}$.
   Then, the additive Schwarz preconditioner $B_{\DG}$ defined by \eqref{eq:dg-as} satisfies
   \[
      \kappa(B_{\DG} A_{\IP}) \leq C,
   \]
   where the constant $C$ is independent of $h$, $p$, and penalty parameter $\eta$.
\end{thm}

We now extend this result to the BR2 discretization given by \eqref{eq:br2}.
We begin by defining the interior penalty and BR2 norms,
\begin{align}
   \| v \|_{\IP}^2 &= \| \nabla v \|_{L^2(\Omega)}^2
      + \sum_e \frac{p^2}{h} \| \eta^{1/2} \llb v \rrb \|_{L^2(e)}^2, \\
   \| v \|_{\BR}^2 &= \| \nabla v \|_{L^2(\Omega)}^2
      + \sum_e \| \eta^{1/2} r_e(\llb v \rrb) \|_{L^2(\Omega)}^2,
\end{align}
where we recall that the lifting operator $r_e$ is defined by
\begin{equation} \label{eq:lift}
   \int_\Omega r_e(\bm \varphi) \cdot \bm \tau \, dx
      = - \int_e \bm \varphi \cdot \{ \bm \tau \} \, ds.
\end{equation}
Then, we have the following boundedness and coercivity estimates.
\begin{lem}
   For the IP bilinear form, we have (e.g.\ from \cite{Antonietti2010})
   \begin{align*}
      a_{\IP}(u,v) &\lesssim \| u \|_{\IP} \| v \|_{\IP}, \\
      a_{\IP}(u,u) &\gtrsim \| u \|_{\IP}^2.
   \end{align*}
   For the BR2 bilinear form, we have (Cf.\ \cite{Brezzi2000})
   \begin{align*}
      a_{\BR}(u,v) &\lesssim \| u \|_{\BR} \| v \|_{\BR}, \\
      a_{\BR}(u,u) &\gtrsim \| u \|_{\BR}^2. \\
   \end{align*}
\end{lem}
Thus, to prove an estimate for the BR2 bilinear form, it suffices to prove equivalence of the norms $\|\cdot\|_{\IP}$ and $\|\cdot\|_{\BR}$.
We will make use of the following results from \cite{Burman2007}.
\begin{lem}[{[\citenum{Burman2007}, Lemma 3.1]}]
   The following trace and inverse trace inequalities hold.
   \begin{alignat}{2}
      \label{eq:trace-1}
      \|v\|_{L^2(\partial K)}^2 &\lesssim \frac{p^2}{h} \| v \|_{L^2(K)}^2
      \qquad &&\text{for all $v \in V_{\DG}$,} \\
     \label{eq:trace-2}
      \|v_B\|_{L^2(K)}^2 &\lesssim \frac{h}{p^2} \| v_B\|_{L^2(\partial K)}^2
      \qquad &&\text{for all $v_B \in V_{B}$,}
   \end{alignat}
   where we recall that the space $V_{B}$ denotes the space of functions in $V_{\DG}$ that vanish on all interior Gauss-Lobatto nodes.
\end{lem}
\begin{prop}
   The norms $\|\cdot\|_{\IP}$ and $\|\cdot\|_{\BR}$ are equivalent.
\end{prop}
\begin{proof}
It is clear that it suffices to prove for every face $e$ that
\begin{align*}
   \| r_e(\llb v \rrb) \|_{L^2(\Omega)}^2 &\lesssim
      \frac{p^2}{h} \| \llb v \rrb \|_{L^2(e)}^2, \\
   \frac{p^2}{h} \| \llb v \rrb \|_{L^2(e)}^2 &\lesssim
       \| r_e(\llb v \rrb) \|_{L^2(\Omega)}^2.
\end{align*}
Let $\llb v \rrb$ be given on $e$.
Let $\bm v$ denote the extension of $\llb v \rrb$ to the element $K$ by setting $\bm v = \llb v \rrb$ componentwise on all Gauss-Lobatto nodes lying on the face $e$, and $\bm v = 0$ at all other Gauss-Lobatto nodes.
Then, choosing $\bm\tau = \bm v$ and $\bm\varphi = \llb v \rrb$ in equation \eqref{eq:lift}, we have
\[
   \| \llb v \rrb\|_{L^2(e)}^2 = - \int_\Omega r_e(\llb v \rrb) \cdot \bm v \,dx
      \leq \| r_e(\llb v \rrb) \|_{L^2(\Omega)} \| \bm v \|_{L^2(\Omega)}
      \lesssim \frac{h^{1/2}}{p}\| r_e(\llb v \rrb) \|_{L^2(\Omega)}
          \| \llb v \rrb \|_{L^2(e)},
\]
where for the last inequality we made use of \eqref{eq:trace-2}.

For the other direction, we let $\bm\tau = r_e(\llb v \rrb)$ and $\bm\varphi =
\llb v \rrb$, and obtain from \eqref{eq:lift}
\[
   \| r_e(\llb v \rrb) \|_{L^2(\Omega)}^2
      \leq \| \llb v \rrb \|_{L^2(e)} \| r_e(\llb v \rrb) \|_{L^2(e)}
      \lesssim \frac{p}{h^{1/2}} \| \llb v \rrb \|_{L^2(e)}
            \| r_e(\llb v \rrb) \|_{L^2(\Omega)},
\]
making use of inequality \eqref{eq:trace-1}.
\end{proof}

\begin{cor}
   Let $A_{\BR}$ denote the operator corresponding to the BR2 bilinear form $a_{\BR}$.
   Then, the additive Schwarz preconditioner $B_{\DG}$ defined by \eqref{eq:dg-as} satisfies
   \[
      \kappa(B_{\DG} A_{\BR}) \leq C,
   \]
   where the constant $C$ is independent of $h$, $p$, and penalty parameter $\eta$.
\end{cor}

In Section \ref{sec:numerical-dg}, we present numerical results using both the IP and BR2 discretizations, verifying the robustness of the resulting preconditioner.

\subsection{Computational cost and memory requirements}

We briefly discuss the computational cost associated with the matrix-free preconditioners described above.
Recall that the number of degrees of freedom for the high-order problem \eqref{eq:system} scales like $\ndof = \O(p^d \nel)$, where $\nel$ is the number of elements in the mesh $\T_p$.
Using sum-factorized operator evaluation, the action of the operator $K_p$ can be computed in $\O(p^{d+1}\nel) = \O(p\ndof)$ operations \cite{Orszag1980, Pazner2018}.
The matrices associated with the low-order refined discretization are assembled, requiring a constant number of operations per degree of freedom.
Similarly, because of the inherent sparsity of the low-order operator, the number of nonzeros per row of these matrices is constant, independent of $p$.
Let $\ncoarse$ denote the number of degrees of freedom in the coarse (multilinear) finite element space defined on the mesh $\T_p$.
It is clear that $\ncoarse < \ndof$, and in the very high-order case, we have $\ncoarse \ll \ndof$.
Additionally, let $\npatch$ denote an upper bound on the number of degrees of freedom in a given patch and recall that $\npatches$ denotes the total number of mesh patches, so that $\ndof \lesssim \npatches\npatch$.
We will assume that the number of elements per patch is bounded by some constant, and so $\npatch = \O(p^d)$.
The minimum discarded fill ordering on each patch is performed in $\O(\npatch \log(\npatch))$ time using a min-heap data structure \cite{Persson2008}.
This cost is asymptotically dominated by the cost of the high-order operator evaluation, and is only performed once as a preprocessing step, rather than at every iteration.
The computational cost and memory requirements for the preconditioning operations are summarized in Table \ref{tab:complexity}.
Given that the preconditioner is robust in the discretization parameters, we conclude that the total runtime for the algorithm scales like the cost of a matrix-free application of the high-order operator.
Furthermore, the total memory cost associated with building and applying the preconditioner is optimal, scaling linearly in the number of degrees of freedom.

\begin{table}[h]
   \centering
   \caption{Computational complexity and memory requirements for
            preconditioning operations}
   \label{tab:complexity}
   \begin{tabular}{r @{\hskip 24pt} ll}
      \toprule
       & Operations & Memory \\
      \midrule
      Matrix-free application of $K_p$ & $\O(p\ndof)$ & $\O(\ndof)$ \\
      Matrix-based assembly of $K_h$ & $\O(\ndof)$ & $\O(\ndof)$ \\
      Coarse solver $R_0$ & $\O(\ncoarse)$  & $\O(\ncoarse)$ \\
      MDF ordering & $\O(\ndof\log(\npatch))$  & $\O(\ndof)$ \\
      ILU smoothing & $\O(\ndof)$  & $\O(\ndof)$ \\
      \midrule
      Total preconditioner application & $\O(p \ndof)$ & $\O(\ndof)$ \\
      \bottomrule
   \end{tabular}
\end{table}

\section{Numerical results} \label{sec:numerical}

In this section, we present a variety of numerical results intended to study the performance of the preconditioner and verify the properties shown in the preceding sections.
The preconditioner was implemented in the MFEM finite element discretization library \cite{Anderson2019}.

\subsection{Cartesian grid}

For a first numerical example, we consider the $h$- and $p$-refinement of a Cartesian grid with constant coefficient $b = 1$.
We begin with a $2 \times 2$ grid, and refine uniformly by factors of two, so that the finest mesh is a $32 \times 32$ grid.
Polynomial degrees between 2 and 20 are used, such that the largest system consider has a total of 410{,}881 degrees of freedom.
We first study the performance of the local multigrid solver by using a single additive Schwarz patch for the entire domain.
This corresponds to preconditioning the high-order system $K_p$ with an element-structured multigrid V-cycle applied to the low-order refined system $K_h$, coarsening each element using the strategy described in Section \ref{sec:multigrid}.
The coarsening procedure terminates when all interior Gauss-Lobatto points have been eliminated.
One application of MDF-ordered ILU(0) is used for pre-smoothing and post-smoothing at each level.
The coarse solve is performed on the $p=1$ discretization on the coarse mesh $\T_p$, using one V-cycle of BoomerAMG.
The number of conjugate gradient iterations required to reduce the residual by a factor of $10^{8}$ are shown in Table \ref{tab:cart-iters}.
We observe that for this problem, fewer than 20 iterations were required for all cases considered.
We also observe a mild pre-asymptotic dependence of the number of iterations on $p$ for low polynomial degrees.
However, for $p$ larger than about 10, the iterations remain almost constant.

Additionally, we study the performance of the additive Schwarz method with vertex patches.
In this case, the domain is decomposed into overlapping subdomains corresponding to each vertex of the coarse mesh $\T_p$.
Each subdomain consists of all elements containing the given vertex.
The same geometric multigrid procedure with ILU smoothing as described above is applied to each of the low-order refined subdomain local problems independently.
As before, the coarse solver is one V-cycle of BoomerAMG on the global coarse problem.
The number of conjugate gradient iterations required to reduce the residual by a factor of $10^{8}$ are shown in Table \ref{tab:cart-iters-patch}.
Similar to the previous case, we observe a pre-asymptotic dependence on the mesh size and polynomial degree.
The number of iterations in larger by about a factor of two when compared with one patch per domain.
However, in the case of vertex patches, the computations can be performed in parallel on each patch.

Note that the iteration counts shown in Tables \ref{tab:cart-iters} and \ref{tab:cart-iters-patch} are largely consistent with past results from the literature on low-order preconditioning of high-order operators.
In particular, we compare these results to the low-order finite element preconditioning of spectral element methods presented in \cite{Lottes2005}, in which several additive Schwarz and multigrid preconditioners were studied.
The results from \cite{Lottes2005} suggest that increased performance could be obtained by using a hybrid Schwarz strategy (incorporating the coarse solve in a multiplicative rather than additive way), and by using a larger coarse space.
Additionally, in \cite{Lottes2005}, the importance of weighting the smoother by a diagonal counting matrix was emphasized.
Investigating these considerations in the context of the present solver is future work.

\begin{table}
\centering
\caption{Number of CG iterations required to reduce the residual by a factor of $10^{8}$. Constant coefficient Cartesian grid with one patch.}
\label{tab:cart-iters}
\begin{tabular*}{0.99\textwidth}{l|@{\extracolsep{\fill}}ccccc}
\toprule
$p$ & $n_x=2$ & $n_x=4$ & $n_x=8$ & $n_x=16$ & $n_x=32$ \\
\midrule
2 & 4 & 10 & 12 & 12 & 13\\
4 & 13 & 14 & 14 & 14 & 14\\
6 & 15 & 16 & 15 & 16 & 16\\
8 & 16 & 16 & 16 & 15 & 16\\
10 & 17 & 17 & 17 & 17 & 18\\
12 & 17 & 18 & 17 & 18 & 19\\
14 & 18 & 18 & 17 & 19 & 19\\
16 & 18 & 17 & 17 & 16 & 18\\
18 & 18 & 18 & 17 & 19 & 19\\
20 & 18 & 18 & 17 & 19 & 19\\
\bottomrule
\end{tabular*}

\vspace{\floatsep}

\caption{Number of CG iterations required to reduce the residual by a factor of $10^{8}$. Constant coefficient Cartesian grid with vertex patches.}
\label{tab:cart-iters-patch}
\begin{tabular*}{0.99\textwidth}{l|@{\extracolsep{\fill}}ccccc}
\toprule
$p$ & $n_x=2$ & $n_x=4$ & $n_x=8$ & $n_x=16$ & $n_x=32$ \\
\midrule
2 & 4 & 10 & 14 & 20 & 26\\
4 & 12 & 17 & 22 & 26 & 29\\
6 & 17 & 22 & 26 & 31 & 32\\
8 & 19 & 24 & 28 & 33 & 34\\
10 & 22 & 26 & 31 & 35 & 36\\
12 & 24 & 27 & 30 & 35 & 36\\
14 & 24 & 30 & 32 & 36 & 37\\
16 & 25 & 31 & 33 & 36 & 37\\
18 & 26 & 31 & 34 & 37 & 38\\
20 & 27 & 32 & 34 & 37 & 38\\
\bottomrule
\end{tabular*}
\end{table}

We also study the performance of the preconditioner in terms of wall-clock run time under $h$- and $p$-refinement in 2D and 3D.
We first fix the polynomial degree to be $p=8$, and perform a sequence of uniform refinements.
The wall-clock time required for the solution (i.e.\ total time over all CG iterations) scales optimally (i.e.\ $\mathcal{O}\left(\ndof\right)$).
The time required for setup (i.e.\ construction of the preconditioner) scales like $\mathcal{O}\left( \ndof \log(\ndof) \right)$, because of the $n \log(n)$ complexity of the MDF ordering used in the ILU factorization.
In 2D, the cost of the MDF factorization appears to be negligible, and we observe the setup run time to scale close to linearly in the number of degrees of freedom.
However, in 3D, the cost of computing the MDF ordering is no longer negligible.
We also perform $p$-refinement on a fixed mesh in 2D and 3D.
In both cases, we observe that the setup run time scales linearly in the number of degrees of freedom ($\mathcal{O}(p^d)$) since the construction of the preconditioner requires only the low-order operator $K_h$.
The run time for the CG iterations scales like $\mathcal{O}(p^{d+1})$ because of the cost of applying the high-order operator $K_p$ using the sum factorization technique.
Figures \ref{fig:runtimes-h} and \ref{fig:runtimes-p} show timings performed in serial running on an Intel Xeon Gold 6130 CPU (2.10 GHz), confirming these scalings.

\begin{figure} \centering
\includegraphics[width=0.5\linewidth]{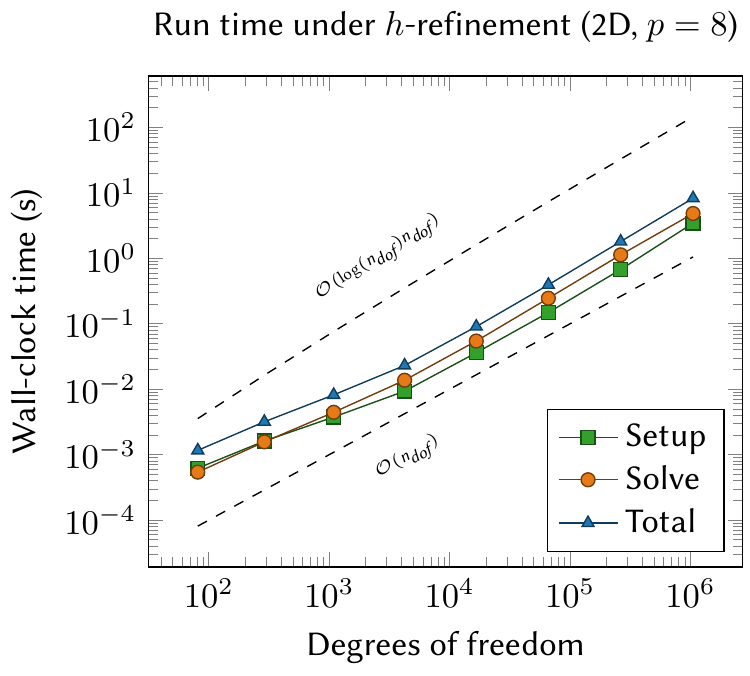}%
\includegraphics[width=0.5\linewidth]{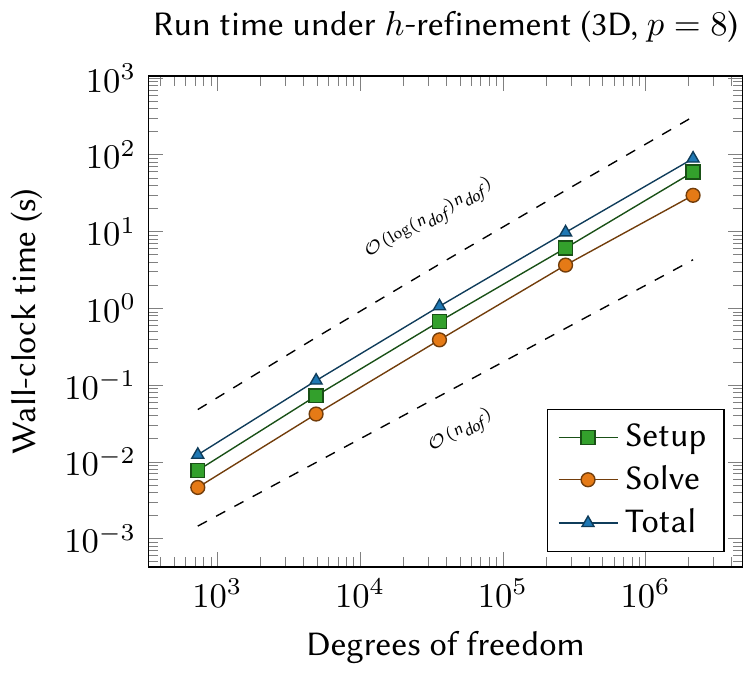}
\caption{Wall-clock run times in 2D and 3D under uniform $h$-refinement using
         fixed polynomial degree $p=8$.}
\label{fig:runtimes-h}

\vspace{\floatsep}

\includegraphics[width=0.5\linewidth]{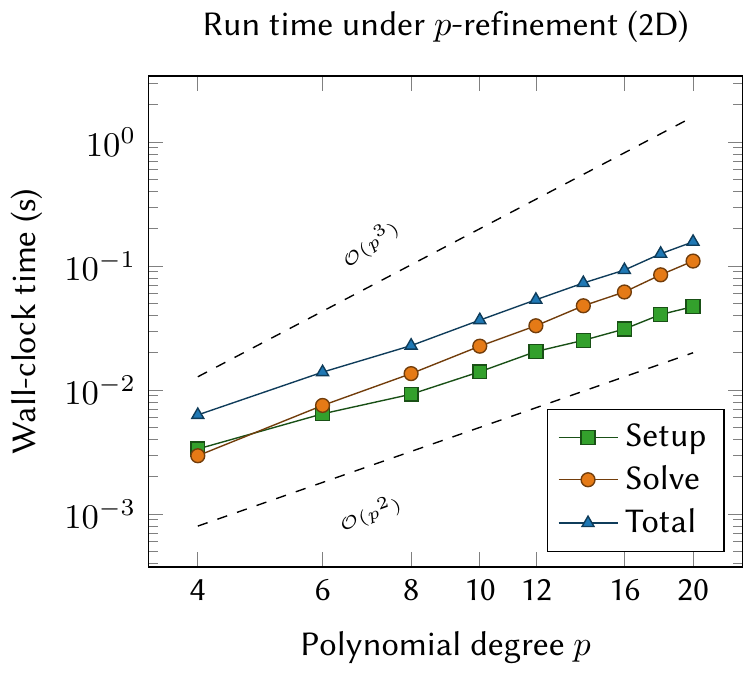}%
\includegraphics[width=0.5\linewidth]{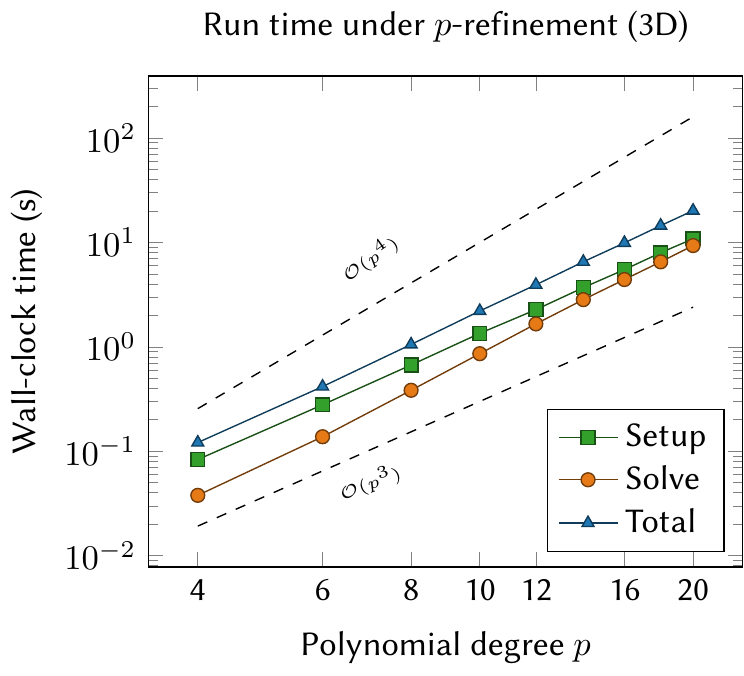}
\caption{Wall-clock run times in 2D and 3D under $p$-refinement using
         a fixed Cartesian mesh.}
\label{fig:runtimes-p}
\end{figure}

\subsection{Variable coefficients and unstructured meshes}
\label{sec:variable-coeff}

We now consider the variable coefficient case, with several choices of non-constant function $b(x,y)$.
The coefficients (the first three of which are from \cite{Shen2000}) are:
\begin{itemize}
\item $b_1(x,y) = 10^4(1-x^2)(1-y^2)$, which features sharp gradients near the
      boundaries of the domain $[-1,1]^2$.
\item $b_2(x,y) = 100x^2 + y^2 + 1$, which is anisotropic, and results in an
      anisotropic solution with a steep gradient.
\item $b_3(x,y) = (1 + x^2 + y^2)^4$.
\item $b_4(x,y)$ is chosen to be a discontinuous, piecewise constant coefficient
      with values 10 and 1 on elements, distributed randomly on the mesh.
\end{itemize}
We also consider unstructured meshes and more complicated geometries.
The meshes used for the numerical experiments are shown in Figure \ref{fig:meshes}.
The first mesh is an unstructured quadrilateral mesh of the domain $[-1,1]^2$.
The domain of the second mesh consists of the unit square $[-1,1]^2$ with a disk of radius $1/4$ removed.
The iteration counts remain bounded for all coefficients, although convergence is slower for $b_1$ than for the remaining coefficients.
The additive Schwarz method with one patch per vertex converges in about twice as many iterations as using a single patch for the domain.
Iteration counts are similar between the two meshes, and are about 50\% larger than in the case of the Cartesian grid.

\begin{figure}
\begin{minipage}{0.49\linewidth}\centering\small
\includegraphics[width=0.75\linewidth]{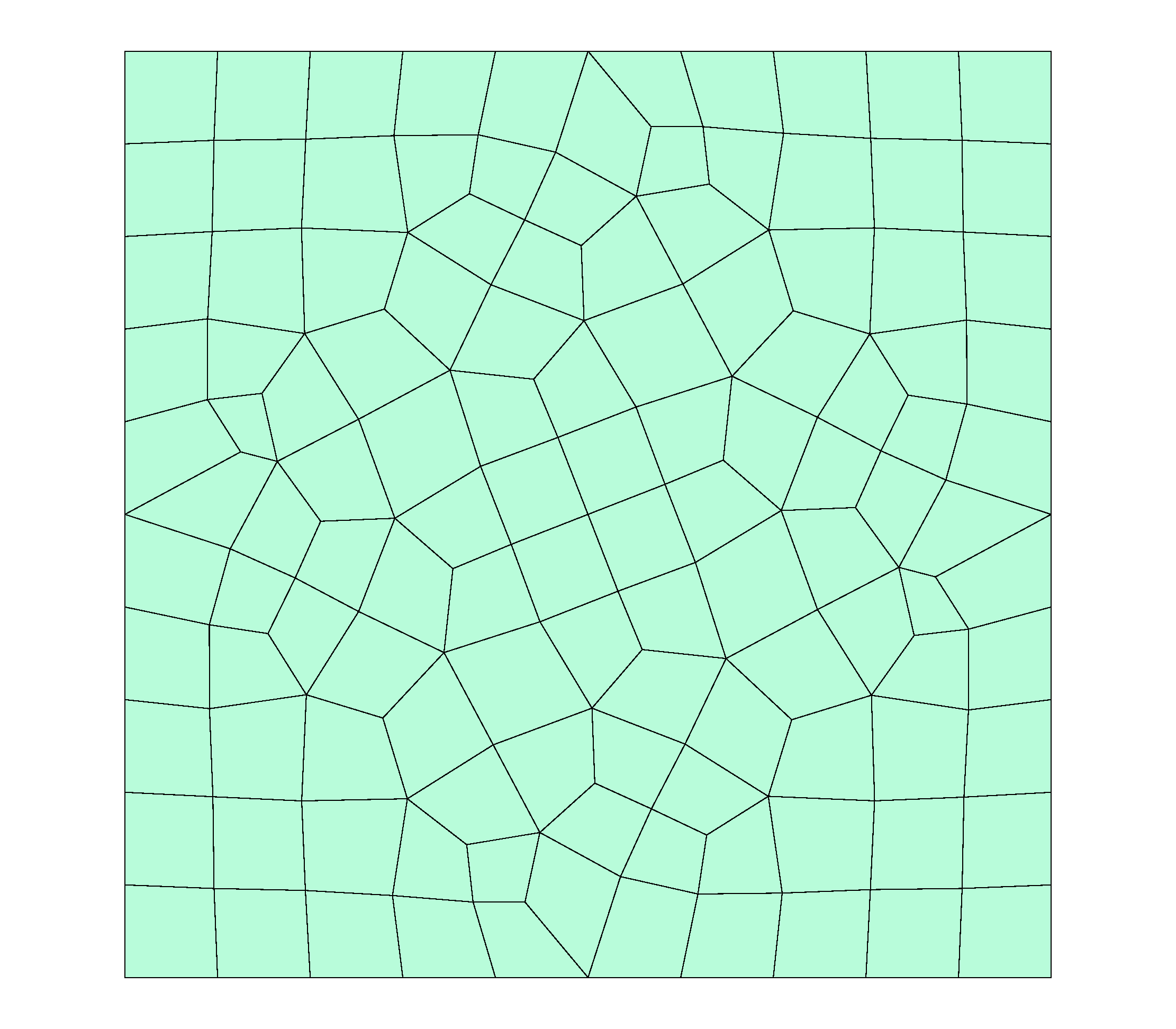}\par(a)
\end{minipage}
\hfill
\begin{minipage}{0.49\linewidth}\centering\small
\includegraphics[width=0.75\linewidth]{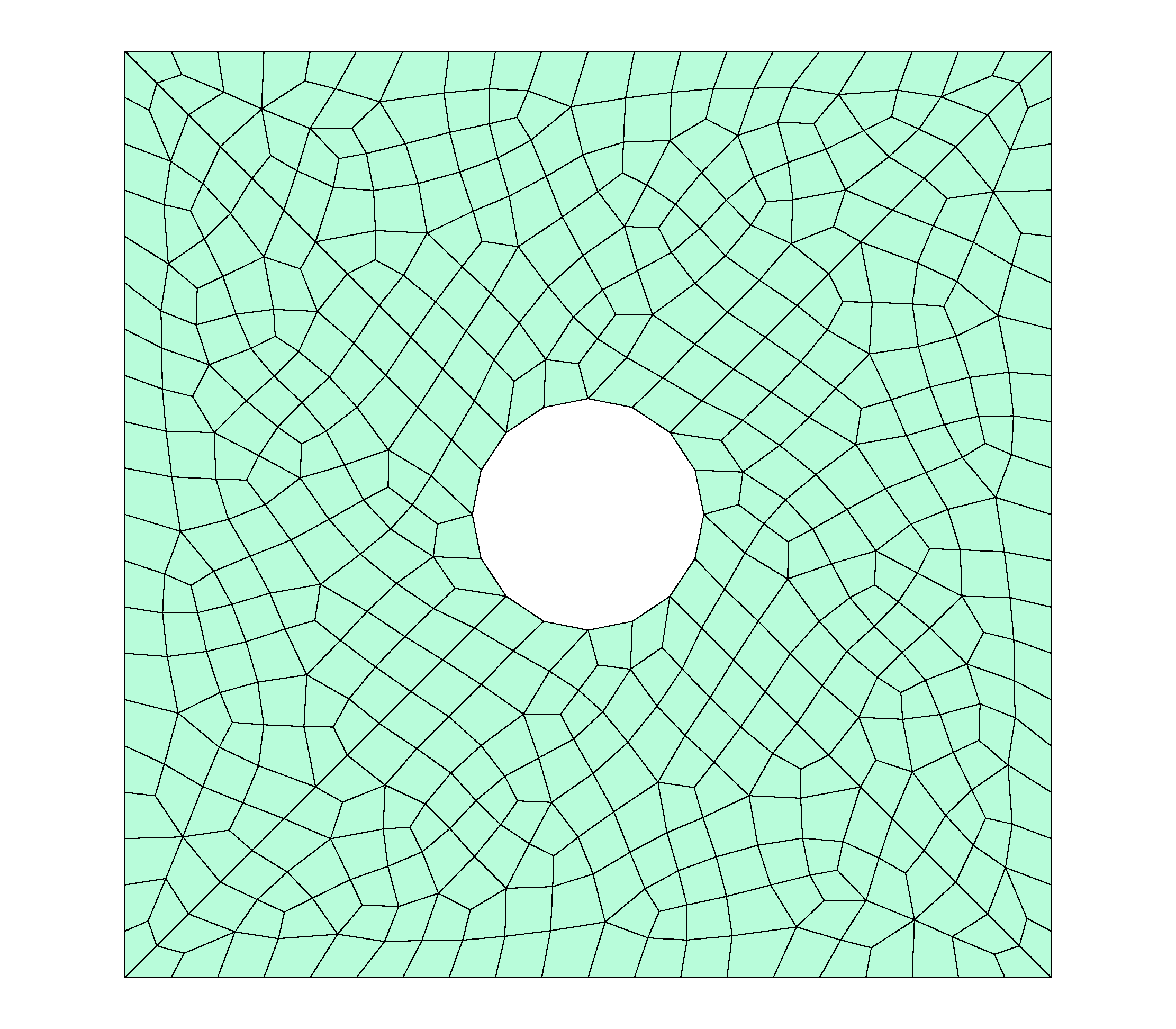}\par(b)
\end{minipage}
\caption{Unstructured quadrilateral meshes used for the numerical experiments.}
\label{fig:meshes}

\vspace{\floatsep}

\centering
\includegraphics{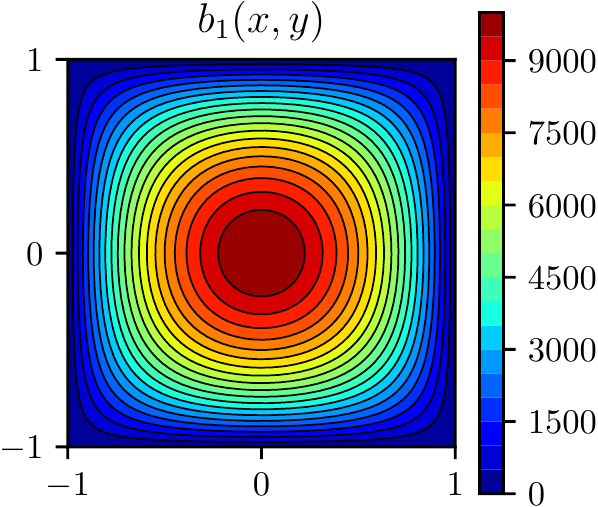} \hspace{1cm}
\includegraphics{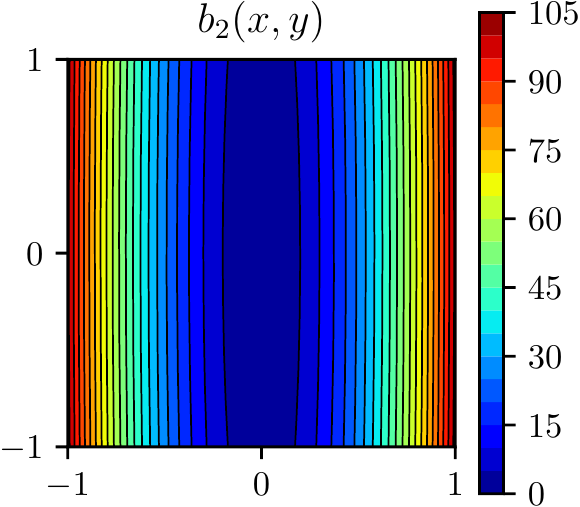}
\includegraphics{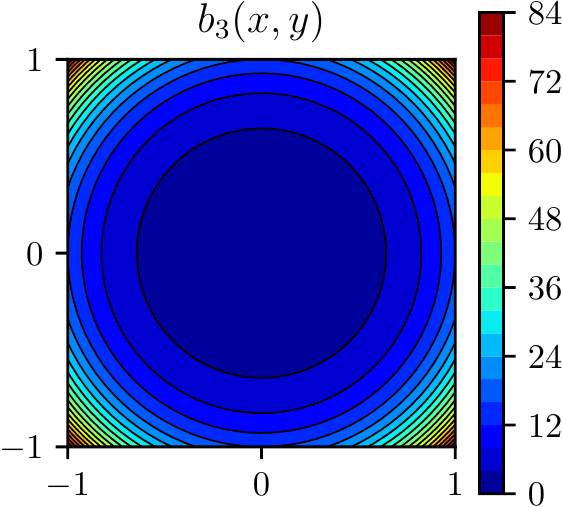}
\caption{Variable coefficients $b_i(x,y)$ used for the examples in Section
         \ref{sec:variable-coeff}.}
\label{fig:coeffs}
\end{figure}

\begin{table}
\centering
\caption{Variable coefficient case on unstructured mesh (a). Numer of iterations
         required to reduce the residual by $10^8$.}
{\setlength{\tabcolsep}{8pt}
\begin{tabular}{l|cccc@{\hskip 24pt}|@{\hskip 24pt}cccc}
\toprule
      \parbox[t][11pt]{0pt}{}
    & \multicolumn{4}{c@{\hskip 24pt}|@{\hskip 24pt}}{Single patch}
    & \multicolumn{4}{@{}c}{Vertex patches} \\
$p$ & $b_1(x,y)$ & $b_2(x,y)$ & $b_3(x,y)$ & $b_4(x,y)$
    & $b_1(x,y)$ & $b_2(x,y)$ & $b_3(x,y)$ & $b_4(x,y)$ \\
\midrule
2 & 15 & 15 & 14 & 16 & 24 & 23 & 22 & 26\\
4 & 18 & 17 & 17 & 17 & 35 & 30 & 29 & 31\\
6 & 23 & 19 & 19 & 19 & 42 & 34 & 34 & 35\\
8 & 23 & 21 & 21 & 21 & 49 & 38 & 37 & 38\\
10 & 29 & 22 & 22 & 23 & 55 & 40 & 40 & 41\\
12 & 30 & 22 & 22 & 23 & 58 & 41 & 41 & 43\\
14 & 35 & 23 & 23 & 26 & 63 & 43 & 43 & 44\\
16 & 30 & 24 & 24 & 25 & 66 & 44 & 44 & 46\\
18 & 41 & 25 & 25 & 28 & 71 & 46 & 45 & 47\\
20 & 37 & 24 & 24 & 28 & 72 & 46 & 46 & 48\\
\bottomrule
\end{tabular}

\vspace{\floatsep}

\caption{Variable coefficient case on unstructured mesh (b). Numer of iterations
         required to reduce the residual by $10^8$.}
\begin{tabular}{l|cccc@{\hskip 24pt}|@{\hskip 24pt}cccc}
\toprule
      \parbox[t][11pt]{0pt}{}
    & \multicolumn{4}{c@{\hskip 24pt}|@{\hskip 24pt}}{Single patch}
    & \multicolumn{4}{@{}c}{Vertex patches} \\
$p$ & $b_1(x,y)$ & $b_2(x,y)$ & $b_3(x,y)$ & $b_4(x,y)$
    & $b_1(x,y)$ & $b_2(x,y)$ & $b_3(x,y)$ & $b_4(x,y)$ \\
\midrule
2 & 15 & 15 & 15 & 16 & 31 & 28 & 27 & 31\\
4 & 18 & 17 & 17 & 18 & 39 & 34 & 34 & 33\\
6 & 23 & 20 & 20 & 22 & 46 & 39 & 38 & 37\\
8 & 24 & 22 & 22 & 22 & 52 & 42 & 41 & 41\\
10 & 31 & 23 & 23 & 26 & 56 & 44 & 43 & 44\\
12 & 30 & 24 & 24 & 27 & 60 & 45 & 44 & 46\\
14 & 36 & 25 & 25 & 30 & 66 & 46 & 46 & 50\\
16 & 29 & 25 & 25 & 27 & 66 & 47 & 47 & 50\\
18 & 41 & 27 & 26 & 34 & 72 & 48 & 47 & 55\\
20 & 37 & 26 & 26 & 31 & 71 & 48 & 48 & 53\\
\bottomrule
\end{tabular}
}
\end{table}

\subsection{Discontinuous Galerkin methods} \label{sec:numerical-dg}

To test the performance of the additive Schwarz preconditioner $B_{\DG}$ applied
to interior penalty and BR2 discontinuous Galerkin discretizations, we again
make use of the unstructured meshes shown in Figure \ref{fig:meshes}. We measure
the number of conjugate gradient iterations required to reduce the residual by a
factor of $10^8$ under $h$- and $p$-refinement. We present these results in
Tables \ref{tab:dg-p-refinement} and \ref{tab:dg-h-refinement}. We note that the
number of iterations do not increase as the mesh is further refined, verifying
robustness in $h$. Similar to the continuous Galerkin examples presented
previously, we observe a slight pre-asymptotic increase in the iteration count
with $p$, though the number of iterations remains bounded, verifying robustness
in $p$. Since of the main features of the preconditioner is its robustness in
the IP and BR2 penalty parameters $\eta$, we consider a fixed mesh and set
$p=6$. We increase the value of $\eta$ from 1 to 10{,}000 by factors of 10, and
report the iterations required to converge in Table \ref{tab:dg-penalty}.
Confirming the results from Section \ref{sec:dg-precond}, the iteration counts
remains robust in $\eta$ for both the IP and BR2 methods.

\begin{table}
\centering
\caption{BR2 discontinuous Galerkin method. Number of CG iterations required to
         reduce the residual by a factor of $10^8$ under $p$-refinement.}
\label{tab:dg-p-refinement}
\begin{tabular}{c|cc}
\toprule
$p$ & Mesh (a) & Mesh(b) \\
\midrule
2 & 21 & 22 \\
4 & 22 & 23 \\
6 & 25 & 26 \\
8 & 23 & 26 \\
10 & 26 & 29 \\
12 & 26 & 30 \\
14 & 27 & 31 \\
16 & 26 & 29 \\
18 & 29 & 32 \\
20 & 28 & 31 \\
\bottomrule
\end{tabular}

\vspace{\floatsep}

\caption{BR2 discontinuous Galerkin method. Number of CG iterations required to
         reduce the residual by a factor of $10^8$ under $h$-refinement.}
\label{tab:dg-h-refinement}
\begin{tabular}{c|crcr}
\toprule
\parbox[t][11pt]{0pt}{}
   & \multicolumn{2}{c}{Mesh (a)} & \multicolumn{2}{c}{Mesh (b)} \\
Refinements & Iterations & DOFs & Iterations & DOFs \\
\midrule
0 & 25 & 5{,}880 & 26 & 23{,}814 \\
1 & 24 & 23{,}520 & 26 & 95{,}256 \\
2 & 23 & 94{,}080 & 25 & 381{,}024 \\
3 & 22 & 376{,}320 & 24 & 1{,}524{,}096 \\
\bottomrule
\end{tabular}
\end{table}

\begin{table}
\centering
\caption{Effect of DG penalty parameter $\eta$ on convergence. Number of
         iterations required to reduce the residual by a factor of $10^8$
         using interior penalty and BR2 DG methods on the unstructured meshes
         from Figure \ref{fig:meshes}.}
\label{tab:dg-penalty}
\begin{tabular}{c|llll}
\toprule
\parbox[t][11pt]{0pt}{}
   & \multicolumn{2}{c}{Mesh (a)} & \multicolumn{2}{c}{Mesh (b)} \\
$\eta$ & IP & BR2 & IP & BR2 \\
\midrule
$10^0$ & 29 & 23 & 33 & 25 \\
$10^1$ & 27 & 22 & 28 & 23 \\
$10^2$ & 25 & 18 & 27 & 21 \\
$10^3$ & 23 & 18 & 26 & 21 \\
$10^4$ & 22 & 18 & 24 & 19 \\
\bottomrule
\end{tabular}

\end{table}

\subsection{Comparison with AMG methods}

In order to compare the present method with a possible alternative matrix-free method for high-order operators, we consider a FEM-SEM preconditioner using an algebraic multigrid V-cycle as an approximate solver for the low-order preconditioner.
We make use of BoomerAMG with both $\ell_1$-scaled hybrid symmetric Gauss-Seidel relaxation and $\ell_1$-scaled Jacobi relaxation \cite{Henson2002}.
We compare the number of CG iterations for increasing polynomial degree $p$.
For this test, we use an unstructured mesh with 80 curved isoparametric quadrilateral elements.
We also consider the high-order DG BR2 system with $\eta = 10$ to ensure coercivity of the system.
Because the FEM-SEM equivalence is less straightforward when applied to DG systems, we compare with AMG applied to the fully assembled high-order DG system (i.e.\ standard matrix-based AMG).
The results are shown in Table \ref{tab:amg}.
We notice that the additive Schwarz preconditioner is robust in $p$ and in the number of mesh patches.
Consistent with the previous results, using vertex patches requires about 1.5 times more iterations than a single patch for the domain.
However, these patches allow for parallelization, because the local solutions on all the patches may be computed simultaneously.
The algebraic multigrid methods are not robust in $p$, likely due to the anisotropy introduced in the low-order refined mesh by the Gauss-Lobatto points.
When applied to discontinuous Galerkin discretizations, the number of CG iterations required when using AMG methods were quite large, consistent with results previously reported in the literature, indicating that multigrid methods may struggle with DG discretizations with lifting operators such as local DG or BR2 \cite{Antonietti2015,Olson2011}.

\begin{table}
\centering
\caption{Comparison with algebraic multigrid methods.
   Number of CG iterations required to reduce the residual by a factor of $10^8$.
   AMG(J) and AMG(GS) indicate BoomerAMG using $\ell_1$-scaled Jacobi and hybrid symmetric Gauss-Seidel, respectively.
   LOR-AMG indicates that the AMG V-cycle is applied to the low-order refined system. $B(1)$ and $B(V)$ indicate the present preconditioner with a single patch, and one patch per vertex, respectively.}
\label{tab:amg}
{\setlength{\tabcolsep}{8pt}
\begin{tabular}{c|cccc|cccc}
\toprule
\parbox[t][11pt]{0pt}{}
   & \multicolumn{4}{c|}{Continuous Galerkin}
   & \multicolumn{4}{c}{Discontinuous Galerkin} \\
$p$ & $B(1)$ & $B(V)$ & LOR-AMG(J) & LOR-AMG(G-S)
    & $B(1)$ & $B(V)$ & AMG(J) & AMG(G-S) \\
\midrule
2 & 15 & 24 & 26 & 20 & 18 & 37 & 186 & 81 \\
4 & 18 & 31 & 39 & 31 & 22 & 46 & 329 & 124 \\
6 & 20 & 35 & 53 & 40 & 23 & 48 & 370 & 145 \\
8 & 21 & 38 & 65 & 46 & 25 & 49 & 361 & 130 \\
10 & 23 & 40 & 76 & 52 & 26 & 51 & 381 & 130 \\
12 & 26 & 42 & 96 & 61 & 27 & 53 & 318 & 122 \\
14 & 28 & 43 & 106 & 71 & 29 & 53 & 344 & 131 \\
16 & 30 & 44 & 114 & 75 & 31 & 54 & 296 & 123 \\
18 & 31 & 45 & 123 & 77 & 34 & 55 & 319 & 124 \\
20 & 33 & 46 & 132 & 87 & 36 & 55 & 287 & 110 \\
\bottomrule
\end{tabular}
}
\end{table}

\subsection{Anisotropic meshes}
\label{sec:aniso}

\begin{figure}
  \centering
  \includegraphics[width=0.8\linewidth]{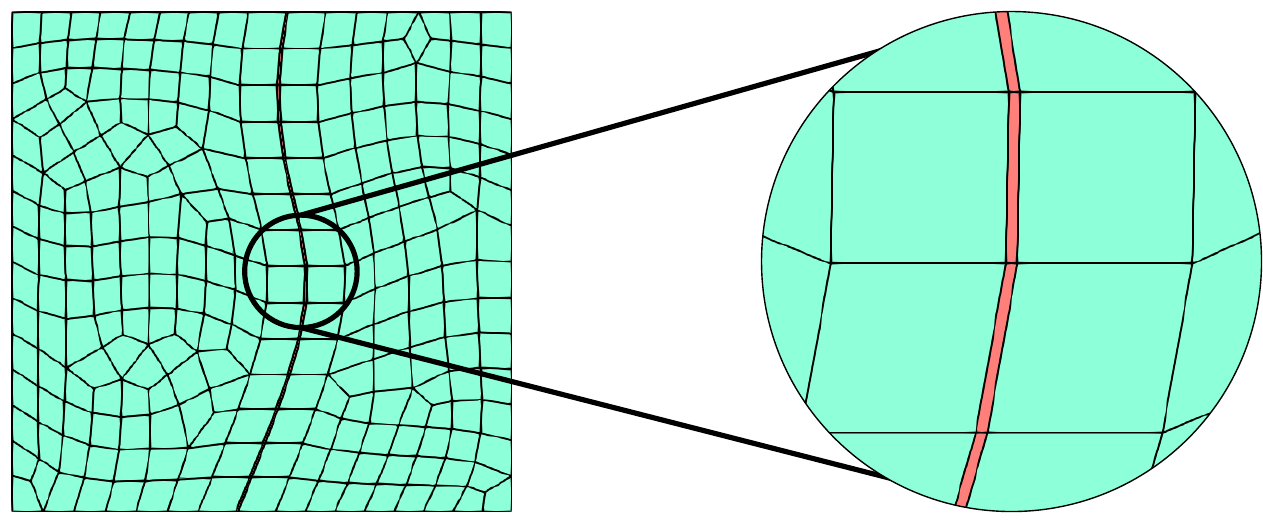}
  \caption{Example of mesh used for the anisotropic test case, showing elements (highlighted in red) with an aspect ratio of 15.}
  \label{fig:aniso-mesh}
\end{figure}

In Remark \ref{rem:anisotropy}, it is mentioned that the additive Schwarz approach described presently could be used with mesh patches featuring anisotropy originating not only from the low-order refined Gauss-Lobatto mesh.
Solvers often struggle with highly anisotropic meshes, which are commonly used to resolve boundary layers when solving convection-dominated problems (e.g.\ high Reynolds number Navier-Stokes) \cite{Mavriplis1998}.
In this example, we consider a sequence of two-dimensional unstructured meshes, with a strip of high-aspect-ratio elements.
We begin with a mesh that is roughly isotropic (maximum aspect ratio 1.5), and increase the aspect ratio of the elements in the strip by a factor of 10, until reaching a maximum aspect ratio of 1500.
An example of one such mesh is shown in Figure \ref{fig:aniso-mesh}.

We solve the high-order constant-coefficient problem (with $p=3,7,11,15$) on this sequence of meshes, using the additive Schwarz preconditioner described here, and compare with BoomerAMG with Gauss-Seidel smoothing applied to the low-order refined problem (denoted LOR-AMR(G-S)).
For this problem, in order for the additive Schwarz method to be robust in the aspect ratio of the elements, the overlap of the subdomains must be independent of the aspect ratio.
To achieve this, subdomain patches containing anisotropic elements are extended to ensure that the overlap remains isotropic.
For this test case, we use $n=1,2,3$ additive Schwarz patches (denoted B($n$)), as well as one patch per vertex (about 250 overlapping patches for this test case), where patches containing anisotropic elements are enlarged to obtain isotropic overlap (denoted B(V)).
In Figure \ref{fig:aniso-iters}, we compare the number of conjugate gradient iterations required to reduce the residual by a factor of $10^8$.
We note that algebraic multigrid applied to the low-order refined system is strongly dependent on both the polynomial degree and on the anisotropy of the mesh.
On the other hand, since the additive Schwarz patches have isotropic overlap, the hypotheses of Lemma \ref{lem:uj} are satisfied, and so the resulting space decomposition is stable.
Consistent with this, the iteration counts using the additive Schwarz solver remain constant with respect to the aspect ratio.
Slight increases in iteration counts are observed for higher polynomial degrees and increased number of subdomains, as is also observed in the preceding numerical examples.

\begin{figure}
  \centering

  \includegraphics[width=0.5\linewidth]{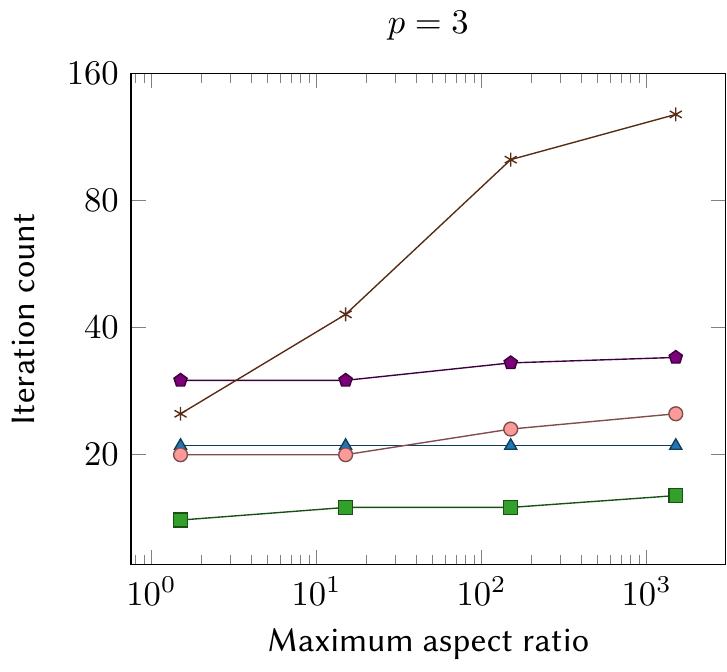}%
  \includegraphics[width=0.5\linewidth]{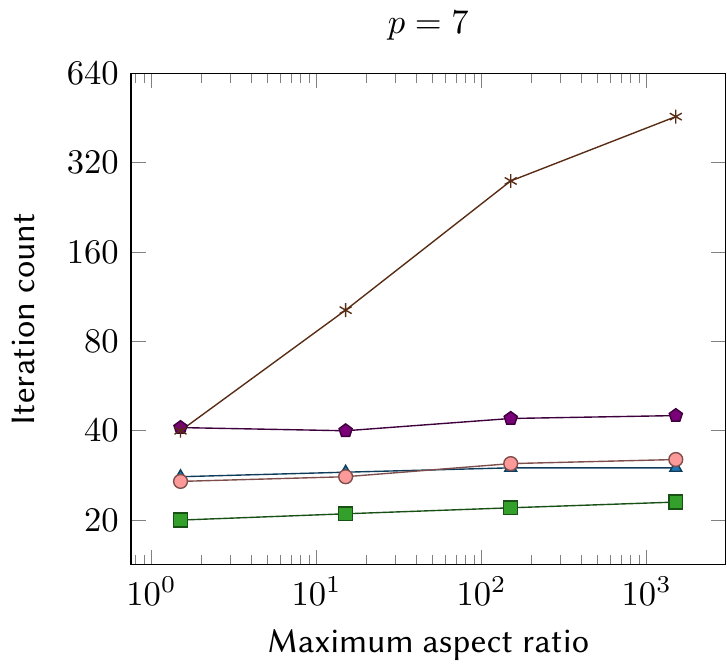}

  \vspace{\floatsep}

  \includegraphics[width=0.5\linewidth]{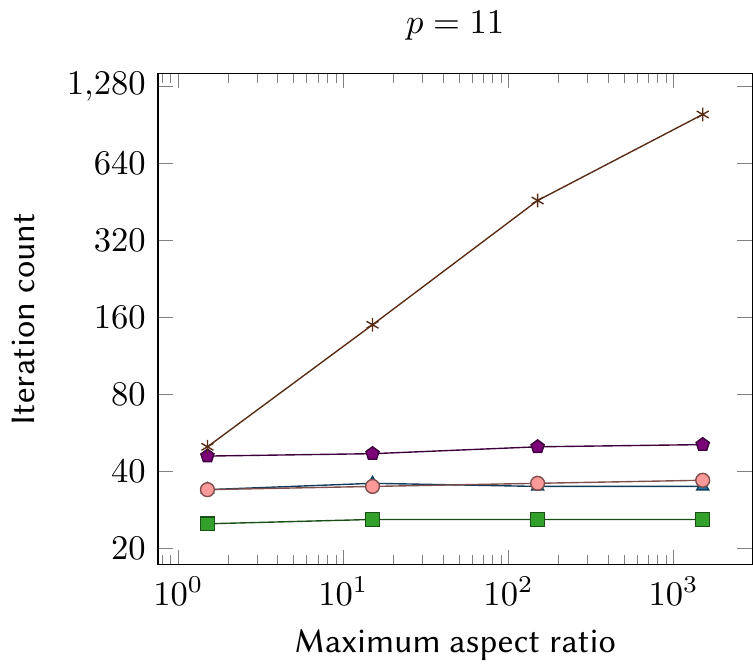}%
  \includegraphics[width=0.5\linewidth]{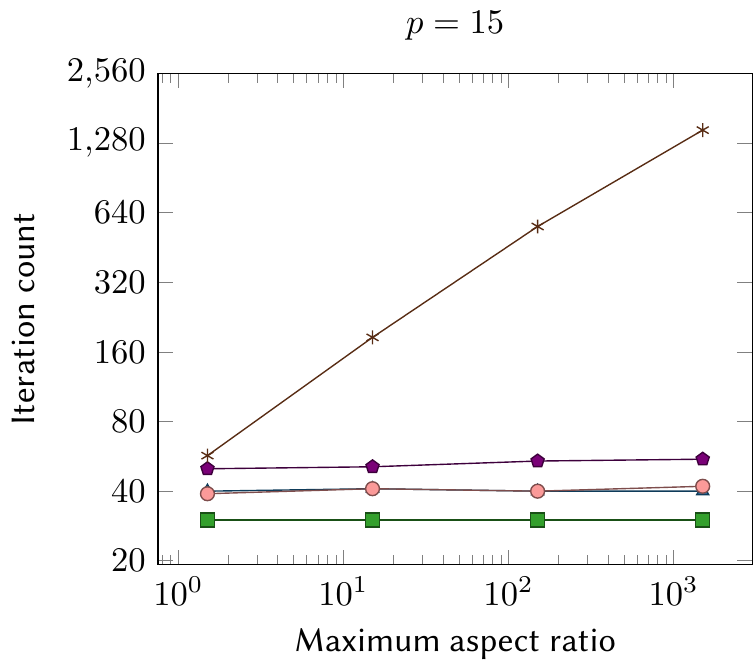}

  \vspace{\floatsep}

  \includegraphics{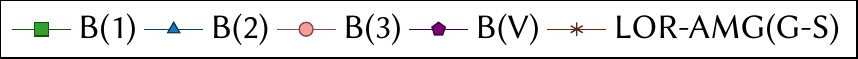}

  \caption{Convergence on a sequence of meshes with anisotropic elements, showing number of conjugate gradient iterations required to reduce the residual by a factor of $10^8$.
  B($n$) indicates additive Schwarz with $n$ subdomains, B(V) indicates additive Schwarz with vertex patches with isotropic overlap, and LOR-AMG(G-S) indicates one V-cycle of BoomerAMG applied to the low-order refined system.}
  \label{fig:aniso-iters}
\end{figure}

\section{Conclusions} \label{sec:conclusions}

In this work, we introduced a matrix-free preconditioner for high-order finite element discretizations of elliptic problems based on a low-order refined methodology.
The low-order refined system is sparse, and the associated matrix can be formed in linear time and memory.
The low-order refined system is preconditioned using a patch-based overlapping additive Schwarz method.
The local approximate solvers in the Schwarz method use a structured geometric multigrid technique with ordered ILU smoothing to treat the anisotropy of the low-order discretization.
The extension to discontinuous Galerkin discretizations is performed naturally in the additive Schwarz context.
The resulting preconditioners are robust in the discretization parameters $h$, $p$, and the DG penalty parameter $\eta$.
Numerical experiments are performed on a variety of geometries and meshes.
The robustness of the preconditioner is verified on test cases involving variable coefficients with sharp gradients and directional biases.
The efficiency of the method is compared with a FEM-SEM preconditioner using algebraic multigrid for the low-order refined system.

\section{Acknowledgements}

The author thanks Tz.\ Kolev and D.\ Kalchev for helpful and insightful conversations.
Lawrence Livermore National Laboratory is operated by Lawrence Livermore National Security, LLC, for the U.S.\ Department of Energy, National Nuclear Security Administration under Contract DE-AC52-07NA27344.
LLNL-JRNL-787240.

\bibliographystyle{siamplain}
\bibliography{refs}

\end{document}